\documentclass[12pt]{amsart}
\usepackage{amssymb,amsmath,amsthm,fullpage,color,soul}
\usepackage{pdfsync}

\newtheorem{thm}{Theorem}[section]
\newtheorem{prop}[thm]{Proposition}
\newtheorem{lem}[thm]{Lemma}

\theoremstyle{definition}
\newtheorem{defn}[thm]{Definition}
\newtheorem{example}[thm]{Example}

\theoremstyle{remark}
\newtheorem{rem}[thm]{Remark}
\newtheorem{remark}[thm]{Remark}
\numberwithin{equation}{section}

\newcommand{\eps}{\ep}

\newcommand{\ddbar}{\partial\bar\partial}
\newcommand{\dom}[1]{\mathrm{Dom}(#1)}

\DeclareMathOperator{\re}{\mathrm{Re}}

\renewcommand{\L}{\mathcal{L}}
\DeclareMathOperator{\dist}{dist}

\newcommand{\X}{\mathcal X}
\DeclareMathOperator{\rea}{Reach}

\newcommand{\Om}{\Omega}
\newcommand{\LL}{\bar L}

\newcommand{\R}{\mathbb R}
\newcommand{\Z}{\mathbb Z}

\newcommand{\N}{\mathbb N}
\newcommand{\C}{\mathbb C}

\DeclareMathOperator{\Imm}{Im}

\DeclareMathOperator{\Rre}{Re}
\DeclareMathOperator{\Dom}{Dom}
\DeclareMathOperator{\Ran}{Range}

\newcommand{\bd}{\textrm{b}}

\newcommand{\p}{\partial}

\newcommand{\z}{\bar z}

\newcommand{\dbar}{\bar\partial}
\newcommand{\dbars}{\bar\partial^*}
\newcommand{\dbarst}{\bar\partial^*_t}
\newcommand{\dbarb}{\bar\partial_b}

\newcommand{\vp}{\varphi}

\newcommand{\atopp}[2]{\genfrac{}{}{0pt}{2}{#1}{#2}}
\newcommand{\nn}{\nonumber}
\newcommand{\ep}{\epsilon}
\newcommand{\I}{\mathcal{I}}

\DeclareMathOperator{\Tr}{Tr}

\newcommand{\la}{\langle}
\newcommand{\ra}{\rangle}

\newcommand{\bom}{\bar\omega}

\begin{document}

\title{Closed range of $\dbar$ in $L^2$-Sobolev spaces on unbounded domains in $\C^n$}

\author{Phillip S. Harrington and Andrew Raich}

\thanks{The second author was partially supported by NSF grant DMS-1405100.}

\address{Department of Mathematical Sciences, SCEN 309, 1 University of Arkansas, Fayetteville, AR 72701}
\email{psharrin@uark.edu \\ araich@uark.edu}



\begin{abstract}
Let $\Om\subset\C^n$ be a domain and $1 \leq q \leq n-1$ fixed. Our purpose in this article is to
establish a general sufficient condition for the closed range of the Cauchy-Riemann operator $\bar\partial$ in appropriately weighted $L^2$-Sobolev spaces
on $(0,q)$-forms.
The domains we consider may be neither bounded nor pseudoconvex, and our condition is a generalization of the classical $Z(q)$ condition that we call weak $Z(q)$.
We provide examples that explain the necessity of working in weighted spaces both for closed range in $L^2$ and, even more critically, in $L^2$-Sobolev spaces.
\end{abstract}

\maketitle

\section{Introduction}
\label{sec:introduction}

This paper is a continuation of \cite{HaRa17}. We
suppose that $\Omega\subset\mathbb{C}^n$ is a smooth domain, and we require neither boundedness nor pseudoconvexity of $\Om$. Our objective to find the weakest possible sufficient condition that ensures
the Cauchy-Riemann operator $\dbar$ has closed range on $(0,q)$-forms in $L^2$-Sobolev spaces, for a \emph{fixed} $q$, $1 \leq q \leq n-1$.  In \cite{HaRa17}, we proved closed range only in $L^2$.
When $\Omega$ is bounded and pseudoconvex,
our result reproduces the classical cases (e.g., Kohn \cite{Koh73}).

We continue to explore the weak $Z(q)$ hypothesis that that we introduced
in \cite{HaRa15}. Weak $Z(q)$ (defined below) is a curvature condition on the Levi form
that suffices to prove that the range of $\dbar$ is closed in $L^2_{0,q}$ or $L^2_{0,q+1}$ on bounded domains in Stein manifolds as well as
unbounded domains with uniform $C^3$ regularity. The weak $Z(q)$ condition is a more general version than the authors' condition in
\cite{HaRa11}, and is closely related to, but still more general than, related conditions in \cite{Ho91}, \cite{AhBaZa06}, and \cite{Zam08} which have been investigated for closed range
of $\dbar$ (or $\dbarb$) in a variety of settings. Its name derives from the fact that it generalizes the
classic $Z(q)$ condition (see \cite{Hor65}, \cite{FoKo72}, \cite{AnGr62}, or \cite{ChSh01}).

Unbounded domains in $\C^n$ may exhibit very different behavior than bounded ones. For example, $\Om$ satisfies
the classic $Z(q)$ condition when the Levi form has either at least $q+1$ negative or at least $n-q$ positive eigenvalues at every boundary point.
However, on any bounded domain, there must be at least one strictly (pseudo)convex boundary point, which forces (by continuity of the eigenvalues of the Levi form) a bounded $Z(q)$ domain in $\mathbb{C}^n$ to
have at least $n-q$ positive eigenvalues at every boundary point.
Hence, a large class of interesting local examples (those with at least $q+1$ negative eigenvalues) cannot be realized globally as bounded domains in $\mathbb{C}^n$ (or indeed any Stein manifold). For an in depth look at
the consequences of $Z(q)$ for unbounded domains, please see \cite{HaRa17z}.

In order to prove closed range of $\dbar$ in $L^2$ on any reasonable class of unbounded domains, it is necessary to work in weighted $L^2$ spaces. Unlike in the bounded case, these weighted $L^2$ spaces are not equivalent to the unweighted
spaces. A simple counterexample demonstrates the necessity of using a weight function.
Suppose that $\Om$ contains balls of arbitrarily large radii. We want to see that the closed range estimate
\begin{equation}\label{eqn:closed range estimate unweighted}
  \|u\|_{L^2(\Om)}\leq C(\|\dbar u\|_{L^2(\Om)}+\|\dbars u\|_{L^2(\Om)})
\end{equation}
cannot hold for any $C>0$. Also
$\dbars$ is the $L^2$ adjoint of $\dbar$ (see Section \ref{sec:weakly_z(q)_domains} for details on the notation).
The Siegel upper space $\{(z,w)\in\C^{n+1}: \Im w > |z|^2\}$ satisfies the large ball condition and is the unbounded domain \emph{par excellence} -- its boundary is the Heisenberg group and it is also biholomorphic to the unit ball.
By the large ball condition, there exists $z_R\in\Omega$ such that $B(z_R,R)\subset\Omega$ for every $R>0$.
Let $u_1\in C^\infty_{0,(0,q)}(B(0,1))$ be nontrivial, and define $u_R(z)=\frac{1}{R^n}u_1\left(\frac{z-z_R}{R}\right)$.  Then $u_R\in C^\infty_{0,(0,q)}(B(z_R,R))\subset C^\infty_{0,(0,q)}(\Omega)$.  If \eqref{eqn:closed range estimate unweighted}
were to hold, then
\[
  \|u_1\|_{L^2(\Om)}=\|u_R\|_{L^2(\Om)}\leq C(\|\dbar u_R\|_{L^2(\Om)}+\|\dbars u_R\|_{L^2(\Om)})=R^{-1}C(\|\dbar u_1\|_{L^2(\Om)}+\|\dbars u_1\|_{L^2(\Om)}).
\]
Since this inequality must hold for every $R>0$, we have a contradiction.  Thus, closed range estimates in $L^2$ are impossible on many unbounded domains, so we must consider weighted $L^2$ spaces.
In \cite{HaRa17}, we do briefly touch upon the $L^2$-theory for $\dbar$ in unweighted $L^2$ spaces for domains that satisfy weak $Z(q)$. Gallagher and McNeal establish
sufficient conditions for the closed range of $\dbar$ in $L^2$ unbounded, pseudoconvex domains \cite{HeMc16}.

Even if we wanted to concentrate on domains for which we can establish the unweighted $L^2$ theory for $\dbar$, there is no hope for any usable result in Sobolev spaces. The reason is that
the Sobolev space theory is effectively useless on any interesting unbounded domain.
For example,
suppose that $\Omega$ contains infinitely many disjoint balls $B_k$ of fixed radius $r$ (as is the case in the model domain defined by $\rho(z)=\sum_{j=1}^n (\re z_j)^2-1$ for which $\dbar$ has closed range in unweighted $L^2$
\cite{HaRa17}).
If we take any function $f\in C^\infty_0(B(0,r))$ and define $f_k(z)=f(z-c_k)$, where $c_k$ is the center of $B_k$, then we have a sequence $\{f_k\}$ that is uniformly bounded in $L^2$ with no convergent subsequence.  Hence,
$H^1(\Om)$ is not compact in $L^2(\Om)$, and the Rellich Lemma fails, making any theory of Sobolev Spaces extremely problematic.

When working on weighted $L^2$ spaces for unbounded domains, adjoints of differential operators can introduce low order terms with unbounded coefficients.  For example, if $D$ is a differential operator and $e^{-\varphi}$ is our weight, we have
\[
  D^*_\varphi=e^{\varphi}D^* e^{-\varphi}=D^*+(\bar D\varphi).
\]
Roughly speaking, our Sobolev spaces must be defined in such a way that multiplying by the unbounded function $\bar D\varphi$ is no worse that differentiating in $D^*$.  This means that great care is required when defining Sobolev spaces.  In \cite{HaRa14}, the authors developed the theory of weighted Sobolev spaces on unbounded domains building on ideas in \cite{GaHa10} and \cite{Gan10}.  Boundary smoothness also requires greater care, since derivatives of defining functions may still be unbounded even when the domain itself is smooth.  In \cite{HaRa13}, the authors carefully examined defining functions for unbounded domains and concluded that from this perspective, the signed distance function works at least as well as any other defining function.
To avoid undue technicalities, we will primarliy use the weight $\varphi=t|z|^2$. Note that $t|z|^2$ will always satisfy $(HII)-(HV)$ in \cite{HaRa14}.

With the tools of \cite{HaRa15}, \cite{HaRa13}, \cite{HaRa14}, and the $L^2$ theory established in \cite{HaRa17}, we are now able to prove closed range of the Cauchy-Riemann operator
on appropriately defined Sobolev spaces
for a large class of unbounded domains.  We review our key definitions in Section \ref{sec:weakly_z(q)_domains}.
Section \ref{sec:basic estimate, proof of L^2 case} recaps the proof of the basic estimate from \cite{HaRa17}.
We conclude the paper with the proof of the main theorem on Sobolev space in Section \ref{sec:weighted_sobolev_spaces}.

\section{Weakly $Z(q)$ domains.}
\label{sec:weakly_z(q)_domains}

\subsection{Notation}
We follow the setup of \cite{HaRa17}.
Let $\Omega\subset \C^n$ be a domain with $C^m$ boundary $\bd\Omega$.
\begin{defn}\label{defn:uniform_defining_function}We say that a defining function $\rho$ for $\Omega$ is \emph{uniformly $C^m$} if there exists an open neighborhood $U$ of $\bd\Omega$ such that $\dist(\bd\Omega,\bd U)>0$, $\|\rho\|_{C^m(U)}<\infty$, and $\inf_U |\nabla\rho|>0$.
\end{defn}
There is no difference between uniform $C^m$ and $C^m$ on domains with compact boundary. On unbounded domains, however,
we provided counterexamples, a large class of examples, and a complete characterization in terms of the signed distance function in \cite{HaRa13}.

We identify real $(1,1)$-forms with a hermitian matrix as follows:
\[
c=\sum_{j,k=1}^{n} i c_{j\bar k}\, dz_j\wedge d\z_k
\]

For a function $\alpha$, we denote $\alpha_k = \frac{\p\alpha}{\p z_k}$ and $\alpha_{\bar j} = \frac{\p\alpha}{\p\z_j}$.

Let $\rho:\C^n\to\R$ be a uniformly $C^m$-defining function for $\Omega$.
The $L^2$-inner product on $L^2(\Omega,e^{-t|z|^2})$ is denoted by
\[
(f,g)_t 
= \int_\Omega f \bar g\, e^{-t|z|^2} dV.
\]
where 
$dV$ is Lebesgue measure on $\C^n$. Let $d\sigma$ denote
the induced surface area measure on $\bd\Omega$ and set $\| f \|_t^2 = \int_\Omega |f|^2 e^{-t|z|^2}\, dV$.

Let $\I_q = \{ (i_1,\dots,i_q)\in \N^n : 1 \leq i_1 < \cdots < i_q\leq n \}$. For $I\in\I_{q-1}$, $J\in\I_q$, and $1\leq j \leq n$, let
$\ep^{jI}_{J} = (-1)^{|\sigma|}$ if $\{j\} \cup I = J$ as sets and $|\sigma|$ is the length of the permutation that takes $\{j\}\cup I$ to $J$. Set
$\ep^{jI}_J=0$ otherwise. We use the standard notation that if $u = \sum_{J\in\I_q} u_J\, d\z_J$, then
\[
u_{jI} = \sum_{J\in\I_q} \ep^{jI}_J u_J.
\]

Let $L^{t}_j = \frac{\p}{\p z_j} - t\z_j = e^{t|z|^2}\frac{\p}{\p z_j} e^{-t|z|^2}$ and let $\dbars_t:L^2_{0,q+1}(\Omega,e^{-t|z|^2})\rightarrow L^2_{0,q}(\Omega,e^{-t|z|^2})$ be the $L^2$-adjoint of
$\dbar:L^2_{0,q}(\Omega, e^{-t|z|^2}) \to L^2_{0,q+1}(\Omega, e^{-t|z|^2})$. This means that if
$f = \sum_{J\in\I_q} f_J\, d\z_J$ and $g = \sum_{K\in\I_{q+1}}g_K\, d\z_K\in\Dom(\dbars_t)$, then
\[
\dbar f = \sum_{\atopp{J\in\I_q}{K\in\I_{q+1}}}\sum_{k=1}^n \ep^{kJ}_K \frac{\p f_J}{\p\z_k}\, d\z_K
\qquad\text{and}\qquad
\dbars_t g = -\sum_{J\in\I_{q}} \sum_{j=1}^n L^t_j g_{jJ}\, d\z_J.
\]

The induced CR-structure on $\bd\Omega$  at $z\in\bd\Omega$ is
\[
T^{1,0}_z(\bd\Omega)  = \{ L\in T^{1,0}(\C) : \p\rho(L)=0 \}.
\]
Let $T^{1,0}(\bd\Omega)$ be the space of $C^{m-1}$ sections of $T^{1,0}_z(\bd\Omega)$ and $T^{0,1}(\bd\Omega) = \overline{T^{1,0}(\bd\Omega)}$.
We denote the exterior algebra generated by these spaces by $T^{p,q}(\bd\Omega)$.

Let $\rho$ be a defining function so that $|d\rho|=1$ on $\bd\Omega$. We define the \emph{normalized Levi form} $\L$ as the real element of $\Lambda^{1,1}(\bd\Omega)$ given by
\[
\L(-i L\wedge \LL) = i\p\dbar\rho(-iL\wedge\LL)
\]
for any $L\in T^{1,0}(\bd\Omega)$.

\begin{defn}\label{defn:tubular nbhd}Given a set $M\subset\C^n$, a \emph{tubular neighborhood} of $M$ is an open set $U_r$ of the form
$U_{r} = \{p\in\C^n : \dist(p,M)<r\}$ where $\dist(\cdot,\cdot)$ is the Euclidean distance function. We call $r$ the \emph{radius} of $U_r$. If there exists $r>0$ so that every point in $U_r$ has a unique closest point in
$M$, we say that $M$ has \emph{positive reach}.
\end{defn}

\subsection{Weak $Z(q)$ domains and closed range for $\dbar$}

The following definition was introduced in \cite{HaRa15}, building on ideas in \cite{HaRa11}.
\begin{defn}\label{defn:weak Z(q)}
Let $\Omega\subset \C^n$ be a domain with a uniformly $C^m$ defining function $\rho$, $m\geq 2$. We say $\bd\Omega$ (or $\Omega$) satisfies
\emph{Z(q) weakly}
if there exists a hermitian matrix $\Upsilon=(\Upsilon^{\bar k j})$ of functions on $b\Omega$ that are uniformly bounded in $C^{m-1}$ such that $\sum_{j=1}^{n}\Upsilon^{\bar k j}\rho_j=0$ on $b\Omega$ and:
\begin{enumerate}\renewcommand{\labelenumi}{(\roman{enumi})}
 \item All eigenvalues of $\Upsilon$ lie in the interval $[0,1]$.

 \item $\mu_1+\cdots+\mu_q-\sum_{j,k=1}^n\Upsilon^{\bar k j}\rho_{j\bar k}\geq 0$ where  $\mu_1,\ldots,\mu_{n-1}$ are the eigenvalues of the Levi form $\L$ in increasing order.

 \item $ \inf_{z\in\bd\Omega} \{ |q-\Tr(\Upsilon)|\} >0$.

\end{enumerate}
\end{defn}
H\"ormander first used an identity, now called the basic identity, to prove a basic estimate for $\dbar$ on pseudoconvex domains \cite{Hor65}. With our current hypotheses, we established the most general basic identity that we could formulate and it
led to the definition of weak $Z(q)$ in \cite{HaRa15,HaRa17}. Given suitable hypotheses, including $f \in\Dom(\dbar)\cap\Dom(\dbars)$, $\bd\Om$ is at least $C^3$, and $i\p\dbar \vp = ti \p\dbar |z|^2$ for some $t\in\R$, the basic identity
we established in \cite[Proposition 3.4]{HaRa17} is
\begin{align}
&\| \dbar f\|_\varphi^2 + \| \dbars_\varphi f \|_\varphi^2
=  \sum_{J\in \I_q}\sum_{j,k=1}^{n}\left((I_{jk}-\Upsilon^{\bar kj})\frac{\partial f_J}{\partial\bar z_k},\frac{\partial f_J}{\partial\bar z_j}\right)_\varphi
+\sum_{J\in \I_q}\sum_{j,k=1}^{n}\left(\Upsilon^{\bar kj}L^\varphi_j  f_J,L^\varphi_k  f_J\right)_\varphi \label{eqn:BI} \\
&+\sum_{I\in\I_{q-1}} \sum_{j,k=1}^n \int_{\bd\Omega}\big\la \rho_{j\bar k} f_{jI}, f_{kI} \big\ra e^{-\varphi}d\sigma
- \sum_{J\in \I_q}\sum_{j,k=1}^n   \int_{b\Omega} \left\la \Upsilon^{\bar kj}\rho_{j\bar k}  f_J , f_J\right\ra e^{-\varphi} d\sigma\nn\\
&+ 2\Rre\Bigg\{  \sum_{J\in \I_q}\sum_{j,k,\ell=1}^{n}\left(\frac{\partial\Upsilon^{\bar kj}}{\p\z_k}\Upsilon^{\bar j\ell}
  L^\varphi_\ell  f_J,f_J\right)_\varphi
  -\sum_{J\in \I_q}\sum_{j,k,\ell=1}^{n} 				
   \left(\frac{\partial\Upsilon^{\bar kj}}{\partial z_j}(I_{k \ell}-\Upsilon^{\bar\ell k})\frac{\partial f_J}{\partial\bar z_\ell},f_J\right)_\varphi \Bigg\}\nn \\
&+ \sum_{J\in\I_q} t \big((q-\Tr(\Upsilon)) f_J, f_J\big)_\varphi + O (\| f \|_\varphi^2),\nn
\end{align}
where $O(\|f\|_{\varphi}^2) \leq C(\|\Upsilon\|_{C^1}+\|\Upsilon\|_{C^2}^2) \| f \|_\varphi^2$ and $I$ is the identity matrix.

The matrix $\Upsilon$ is chosen so that the boundary integral terms in the second line of \eqref{eqn:BI} are nonnegative (and hence can be discarded) while keeping $\inf_{z\in\Om}|q - \Tr\Upsilon| >0$. We also extended $\Upsilon$ into the interior of $\Om$
(Lemma \ref{lem:extension of Upsilon} below). The fact that the eigenvalues of $\Upsilon$ are nonnegative and
bounded by 1 means that the terms in the first line of \eqref{eqn:BI} are nonnegative. Finally, Property (iii) means that $((q-\Tr(\Upsilon)) f, f)_\varphi \sim \|f\|_\vp$, allowing us to prove the basic estimate, Proposition \ref{prop:basic estimate} below.
The constant $t$ is chosen large enough so that the junk terms in the third line of \eqref{eqn:BI} are controlled, as is the $O (\| f \|_\varphi^2)$ term from the fourth line of \eqref{eqn:BI}.

For our results on weighted Sobolev spaces, an additional hypothesis is needed.  In \cite{HaRa14}, we introduced six hypotheses $(HI)-(HVI)$ that were important for developing the
elliptic theory with weighted Sobolev spaces on unbounded domains.  The first hypothesis was equivalent to Definition \ref{defn:uniform_defining_function},
so $(HI)$ will be satisfied whenever we have a uniformly $C^m$ defining function, $m\geq 3$.
Hypotheses $(HII)-(HV)$ are trivial for the weight function $\varphi=t|z|^2$, so we will not need to address them directly in this paper.  Thus, we need only concern ourselves with $(HVI)$.  In the notation of the present paper, we have:
\begin{defn}
\label{defn:asymptotically_non_radial}
  Let $\Omega\subset\mathbb{R}^{d}$ be an unbounded domain.  We say $\Omega$ is \emph{asymptotically non-radial} if
  \[
    \inf_{r>0}\sup_{|x|>r,x\in\bd\Omega}\frac{x\cdot\nabla\rho}{|x||\nabla\rho|}<1
  \]
  for any $C^1$ defining function $\rho$ for $\Omega$.
\end{defn}
In \cite{HaRa14}, this condition is needed in order to show that the restrictions of our weighted Sobolev spaces to $\bd\Omega$ will still satisfy Rellich's Lemma.  A key step in the proof relies on the hypothesis that tangential derivatives of our weight function grow uniformly without bound.  For the special weight function $|x|^2$, this is equivalent to Definition \ref{defn:asymptotically_non_radial}.

Geometrically, we are requiring that the normal vector is bounded away from the radial direction for sufficiently large $|x|$.  To see that this is not a restrictive condition on unbounded domains, observe that $|x|$ can only increase very slowly in the boundary when the normal vector is almost radial.  More precisely, for $r_0>0$ and $0<\theta_1-\theta_0<2\pi$ consider the unbounded open set in polar coordinates $U=\{(r,\theta):r>r_0\text{ and }\theta_0<\theta<\theta_1\}$ and a domain $\Omega\subset\mathbb{R}^2$ defined in polar coordinates on $U$ by $\Omega\cap U=\left\{(r,\theta):r_0<r<e^{f(\theta)},\theta_0<\theta<\theta_1\right\}$ for some $f\in C^1(\theta_0,\theta_1)$.  Since $\Omega$ is defined on $U$ by $\rho(r,\theta)=r-e^{f(\theta)}$, we have $\frac{x\cdot\nabla\rho}{|x||\nabla\rho|}=(1+(f'(\theta))^2)^{-1/2}$.  Hence, $\Omega$ is unbounded and asymptotically nonradial near $\theta_0$ on $U$ if and only if
\[
 \lim_{\theta\rightarrow\theta_0^+}f(\theta)=\infty\text{ and }\limsup_{\theta\rightarrow \theta_0^+}f'(\theta)<0.
\]
Any rational function, for example, would satisfy this property.  Constructing a counterexample that would also define a uniformly $C^2$ domain would require great care.  Although more complicated behavior is possible in higher dimensions, it appears that asymptotic nonradiality is a mild restriction to make on a domain.

Before we state our main result,
we prove a percolation result that greatly expands the scope of our main theorem.
\begin{prop}\label{prop:percolation}
Let $\Omega\subset \C^n$ be a domain with connected boundary that admits a uniformly $C^2$ defining function and
satisfies weak $Z(q)$ for some $1\leq q \leq n-1$. If $q - \Tr(\Upsilon)>0$, then $\Om$ satisfies weak $Z(q')$ for $q \leq q' \leq n-1$. If $q - \Tr(\Upsilon) < 0$, then $\Om$ satisfies weak $Z(q')$ for $1 \leq q' \leq q$.
\end{prop}
\begin{proof}The proof of the proposition follows easily from the fact that we may leave $\Upsilon$ unchanged and \cite[Lemma 2.8]{HaRa15}. This lemma says that weak $Z(q)$ with $q - \Tr\Upsilon>0$ implies that
the Levi form of $\Om$ has at least $(n-q)$ nonnegative eigenvalues and weak $Z(q)$ with $q - \Tr\Upsilon <0$ implies that the Levi form has at least $(q+1)$ nonpositive eigenvalues. The proof becomes transparent
by diagonalizing the Levi form at a point (as we do immediately prior to \cite[Lemma 2.8]{HaRa15}) and inspecting the inequalities from the definition of weak $Z(q)$ in these coordinates.
\end{proof}

The type of estimates that the weighted operators will satisfy is the following: for $t$ sufficiently large, the operator $T_t$, initially known to be bounded from $L^2_{0,q'}(\Om,e^{-t|z|^2})$  to
$L^2_{0,q''}(\Om,e^{-t|z|^2})$ for some $q',q''$, will be shown to be continuous from
$H^s_{0,q'}(\Om,e^{-t|z|^2},X)$ to $H^s_{0,q''}(\Om,e^{-t|z|^2},X)$ and satisfy the estimate
\begin{equation}\label{eqn:Hs,L2 est}
\|T_t u \|_{t,s,\Om}^2 \leq C_s\|u\|_{t,s,\Om}^2 + C_{t,s} \|u\|_t^2
\end{equation}
where $C_s$ only depends on $s$, $C_{t,s}$ depends on both $t$ and $s$, and neither constant depends on $u$.
\begin{thm}\label{thm:closed range for dbar}Let $\Omega\subset \C^n$ be a domain with connected boundary that is asymptotically nonradial, admits a uniformly $C^m$ defining function, $m\geq 3$, has positive reach, and
satisfies weak $Z(q)$ for some $1\leq q \leq n-1$. Let $0\leq s\leq m-2$.
Then there exists a $T_s>0$ so that if $q-\Tr(\Upsilon)>0$ and $t\geq T_s$ or  $q-\Tr(\Upsilon)<0$ and $t\leq -T_s$, then
\begin{enumerate}\renewcommand{\labelenumi}{(\roman{enumi})}
\item The operator $\dbar: H^s_{0,\tilde q}(\Omega, e^{-t|z|^2},X)\to H^s_{0,\tilde q+1}(\Omega, e^{-t|z|^2},X)$ has closed range
for $\tilde q = q-1$ or $q$;
\item The operator $\dbars_t: H^s_{0,\tilde q+1}(\Omega, e^{-t|z|^2},X)\to H^s_{0,\tilde q}(\Omega, e^{-t|z|^2},X)$ has closed range
for $\tilde q = q-1$ or $q$;
\item The weighted $\dbar$-Neumann Laplacian defined by $\Box_{q,t} = \dbar\dbars_t + \dbars_t\dbar$ has closed range on $H^s_{0,q}(\Omega, e^{-t|z|^2},X)$;
\item $\ker(\Box_{q,t}) = \{0\}$.
\item
The following operators are continuous and satisfy estimates of type \eqref{eqn:Hs,L2 est}:
\begin{enumerate}
\item The weighted $\dbar$-Neumann operator $N_{q,t} : H^s_{0,q}(\Omega, e^{-t|z|^2},X) \to H^s_{0,q}(\Omega, e^{-t|z|^2},X)$;
\item $ \dbars_t N_{q,t}:H^s_{0,q}(\Omega, e^{-t|z|^2},X)\to H^s_{0,q-1}(\Omega, e^{-t|z|^2},X)$
\item $N_{q,t}\dbarst : H^s_{0,q+1}(\Omega, e^{-t|z|^2},X) \to H^s_{0,q}(\Omega, e^{-t|z|^2},X)$
\item $\dbar N_{q,t}:H^s_{0,q}(\Omega, e^{-t|z|^2},X)\to H^s_{0,q+1}(\Omega, e^{-t|z|^2},X)$
\item $N_{q,t}\dbar : H^s_{0,q-1}(\Omega, e^{-t|z|^2},X)\to H^s_{0,q}(\Omega, e^{-t|z|^2},X)$
\item $\dbars_t N_{q,t} \dbar: H^s_{0,q-1}(\Omega, e^{-t|z|^2},X) \to H^s_{0,q-1}(\Omega, e^{-t|z|^2},X)$
\item $\dbar \dbars_t N_{q,t}$ and $\dbars_t\dbar N_{q,t}$  mapping $H^s_{0,q}(\Omega, e^{-t|z|^2},X)$ to itself.
\end{enumerate}

\item If $\tilde{q}=q$ or $q+1$ and $\alpha\in H^s_{0,\tilde{q}}(\Omega, e^{-t|z|^2},X)$
so that  $\dbar \alpha =0$, then there exists $u\in H^s_{0,\tilde{q}-1}(\Omega, e^{-t|z|^2},X)$ so that
\[
\dbar u = \alpha.
\]
\item If $m=\infty$, $\tilde{q}=q$ or $q+1$, and $\alpha\in C^\infty_{0,\tilde q}(\overline\Om)$ satisfies
$\dbar \alpha =0$, then there exists $u\in C^\infty_{0,\tilde{q}-1}(\overline\Omega)$ so that
\[
\dbar u = \alpha.
\]
\end{enumerate}
\end{thm}

\begin{remark}The $s=0$ case of Theorem \ref{thm:closed range for dbar} for parts (i) - (viii) is the main result in \cite{HaRa17}. Also,
the operator $\dbars_t N_{q,t}$ is the canonical solution operator for the $\dbar$ equation, and $(\dbar N_{q,t})^*$ is the canonical solution operator for the $\dbar$-equation if
$N_{q+1,t}$ exists. The latter operator may exist as a consequence of Proposition \ref{prop:percolation} and Theorem \ref{thm:closed range for dbar}.
Similarly, the operator $\dbar N_{q,t}$ is the canonical solution operator for $\dbars_t$  on $(0,q)$-forms and $(\dbars_t N_{q,t})^* = N_{q,t}\dbar$ is the canonical solution operator for $\dbars_t$ on $(0,q-1)$-forms if
$N_{q-1,t}$ exists. The operator $N_{q-1,t}$ will exist if $q - \Tr\Upsilon < 0$.
\end{remark}

\begin{remark} We wish to point out a slight errata in Lemma 2.3 from \cite{HaRa13}.  A $C^2$ domain with positive reach must have a uniformly $C^2$ defining function, but the converse is not necessarily true. Consequently, Theorem 2.4 in \cite{HaRa17} needs to include this hypothesis, as it relies on the results from \cite{HaRa13}.
\end{remark}

\begin{example}
In \cite{HaRa17}, we show that for any $1\leq p\leq n-1$ the quadric defined by
\[
  \rho(z)=\sum_{j=1}^p|z_j|^2-\sum_{j=p+1}^n|z_j|^2+1
\]
is a $Z(q)$ domain for any $q\neq n-p-1$ with a uniformly $C^\infty$ defining function.  One can easily check that such domains are also asymptotically non-radial.
\end{example}

\section{The basic estimate}\label{sec:basic estimate, proof of L^2 case}
In this paper, we will use the weight $\vp=t|z|^2$, though we could also consider more general weight functions.
For example, given a generic $C^2$ weight $\varphi$, the final (non-error term) in \ref{eqn:BI} would be
\[
  \sum_{I\in\I_{q-1}}\sum_{j,k=1}^n \big(\varphi_{j\bar k} f_{jI}, f_{kI}\big)_\varphi-\sum_{J\in\I_{q}}\sum_{j,k=1}^n \big(\varphi_{j\bar k}\Upsilon^{\bar k j} f_{J}, f_{J}\big)_\varphi
\]
The price of the more general weight is that we  would have to change (iii) in Definition \ref{defn:weak Z(q)} to
\[
  \inf_{z\in b\Omega}\lambda_1+\cdots+\lambda_q-\sum_{j,k=1}^n\varphi_{j\bar k}\Upsilon^{\bar k j}>0,
\]
where $\lambda_1,\ldots,\lambda_n$ are the eigenvalues of $\varphi_{j\bar k}$ arranged in increasing order (see the definition of $q$-compatible functions in \cite{HaRa11}).
We wish, however, to avoid this technicality. The basic estimate is the content of Proposition \ref{prop:basic estimate}, and it quickly follows from the basic identity \ref{eqn:BI} and the following two
lemmas from \cite{HaRa17}. The first details the extension of $\Upsilon$ into the interior of $\Om$, and the second is a density lemma.

\begin{lem}\label{lem:extension of Upsilon}
Suppose $\Omega$ has a connected boundary, a uniformly $C^m$ defining function for some $m\geq 2$, and satisfies weak $Z(q)$ for some $1\leq q\leq n-1$.
Let $\Upsilon$ be as in Definition \ref{defn:weak Z(q)}. There exists a hermitian matrix
$\tilde\Upsilon$ of functions on $\mathbb{C}^n$ that are uniformly bounded in $C^{m-1}$ satisfying
\begin{enumerate}\renewcommand{\labelenumi}{(\roman{enumi})}
\item All eigenvalues of $\tilde\Upsilon$ lie in the interval $[0,1]$.

\item $\tilde\Upsilon|_{b\Omega}=\Upsilon$, so that $\mu_1+\cdots+\mu_q- \sum_{j,k=1}^n\tilde\Upsilon^{\bar k j}\tilde\delta_{j\bar k}\geq 0$ on $\bd\Omega$ where $\mu_1,\ldots,\mu_{n-1}$ are the eigenvalues of the Levi form in increasing order.

\item $\inf_{z\in\bar\Omega} \{ |q-\Tr(\tilde\Upsilon)|\} >0$.

\item There exists $\ep>0$ so that on the neighborhood $U_\ep$ of $\bd\Omega$ we have
 \begin{equation}\label{eqn:tau rho_k sum=0}
 \sum_{j=1}^n \tilde\Upsilon^{\bar k j}\tilde\delta_j=0.
 \end{equation}
\end{enumerate}
\end{lem}
We do not distinguish between $\Upsilon$ and its extension $\tilde\Upsilon$.

\begin{lem}
\label{lem:bounded_density}
  Let $\Omega\subset\mathbb{C}^n$ be a $C^m$ domain, $m\geq 2$, and let $f\in L^2_{0,q}(\Omega,e^{-\varphi})\cap\dom\dbar\cap\dom{\dbar^*_\varphi}$ for some $C^2$ function $\varphi$.  Then there exists a sequence of bounded $C^m$ domains $\{\Omega_j\}$ and functions $f_j\in C^{m-1}(\Omega)$ such that $\Omega_j\cap B(0,j+2)=\Omega\cap B(0,j+2)$, $f_j\equiv 0$ on $\Omega\backslash B(0,j+2)$, $f_j|_{\Omega_j}\in\dom{\dbar^*_\varphi}$, and
  \begin{multline*}
    \left\|\dbar f_j\right\|_{L^2(\Omega_j,e^{-\varphi})}+\left\|\dbar^*_\varphi f_j\right\|_{L^2(\Omega_j,e^{-\varphi})}+\left\|f_j\right\|_{L^2(\Omega_j,e^{-\varphi})}\\
    \rightarrow \left\|\dbar f\right\|_{L^2(\Omega,e^{-\varphi})}+\left\|\dbar^*_\varphi f\right\|_{L^2(\Omega,e^{-\varphi})}+\left\|f\right\|_{L^2(\Omega,e^{-\varphi})}
  \end{multline*}
\end{lem}

\begin{prop}\label{prop:basic estimate}
Let $\Omega$ have a connected boundary, a uniformly $C^m$ defining function for some $m\geq 2$, and satisfy weak $Z(q)$ for some $1\leq q\leq n-1$.  Suppose $\varphi$ satisfies $\ddbar\varphi=t\ddbar|z|^2$. Then for any
constant $\ep>0$, there exists $T>0$ so that if 
\begin{enumerate}
  \item either $t\leq-T$ and $(q-\Tr\Upsilon) <0$ or $t\geq T$ and $(q-\Tr\Upsilon) >0$, and
  \item $f\in L^2_{0,q}(\Omega,e^{-\varphi})\cap\Dom(\dbar)\cap\Dom(\dbars_\varphi)$,
\end{enumerate}
then
\[
\ep\left(\| \dbar f \|_\varphi^2 + \| \dbars_\varphi f\|_\varphi^2\right) \geq \| f \|_\varphi^2.
\]
\end{prop}

%
%
\section{Solvability of $\dbar$ in $H^s_{(0,q)}(\Omega,e^{-t|z|^2})$}
\label{sec:weighted_sobolev_spaces}

\subsection{Definition of the Sobolev spaces $H^s(\Omega,e^{-t|z|^2})$}
Define the weighted differential operators
\[
X_j^t = \frac{\p}{\p x_j} - 2tx_j = e^{t|z|^2}\frac{\p}{\p x_j} e^{-t|z|^2} ,\quad 1\leq j\leq 2n
\]
and
\[
\nabla_X^t = (X_1^t,\dots, X_{2n}^t).
\]
\begin{defn} \label{defn:sobo space, weighted deriv}
For a nonnegative $k\in\Z$, define the weighted Sobolev space
\[
H^k(\Omega,e^{-t|z|^2},X^t) = \{ f\in L^2(\Omega,e^{-t|z|^2}) : (X^t)^\alpha f \in L^2(\Omega,e^{-t|z|^2}) \text{ for } |\alpha|\leq k \}
\]
where $\alpha = (\alpha_1,\dots,\alpha_{2n})$ is an $2n$-tuple of nonnegative integers and
\[
  (X^t)^\alpha = (X_1^t)^{\alpha_1}\cdots (X_{2n}^t)^{\alpha_{2n}}.
\]
$H^k(\Omega,e^{-t|z|^2},X^t)$ has the norm
\[
\| f \|_{t,k,\Omega}^2=  \sum_{|\alpha|\leq k} \| (X^t)^\alpha f \|_t^2 .
\]
We suppress writing $\Omega$ when the domain is clear.
Also, let
\begin{multline*}
H^k_0(\Omega,e^{-t|z|^2},X^t) \\
= \left\{ g\in H^k(\Omega,e^{-t|z|^2},X^t) : \text{ there exists } \{\psi_\ell\}\subset C^\infty_c(\Omega) \text{ satisfying }\lim_{\ell\rightarrow\infty} \| g-\psi_\ell \|_{t,k}=0\right\}.
\end{multline*}
In other words, $H^k_0(\Omega,e^{-t|z|^2},X^t)$ is the closure of $C^\infty_c(\Omega)$ in the $H^k(\Omega,e^{-t|z|^2},X^t)$-norm.
\end{defn}

For $s>0$, we define $H^s(\Omega,e^{-t|z|^2},X^t)$ by real interpolation. The Sobolev space theory was worked out
by the authors in \cite{HaRa14}. As a consequence of Proposition 3.5 in \cite{HaRa14}, we have the following lemma.

\begin{lem}\label{lem: dom dbar in H^1}
Assume that $\Omega$ is asymptotically non-radial and has a uniformly $C^2$ defining function. Then $H^1_{0,q}(\Omega,e^{-t|z|^2},X^t) \subset \Dom(\dbar)$.
\end{lem}

\subsection{Elliptic regularization.}
Before turning to the proof of Theorem \ref{thm:closed range for dbar} for $s>0$, we need to do some preliminary work.

For $\eps>0$, set
\begin{align*}
Q_t(u,v) &= (\dbar u,\dbar v)_t + (\dbars_t u, \dbars_t v)_t \\
Q_{t,\eps}(u,v) &= (\dbar u,\dbar v)_t + (\dbars_t u, \dbars_t v)_t  + \eps(\nabla_X^t u,\nabla_X^t v)_t
\end{align*}
We can prove the elliptic regularity for $\Box_{t,\eps} = \dbars_t\dbar+\dbar\dbars_t +\eps (\nabla_X^t)^* \nabla_X^t$.
\begin{prop}\label{prop:elliptic regularity for N^q_eps}
Let $\Omega$ satisfy the hypotheses of Theorem \ref{thm:closed range for dbar}.
For $0\leq s \leq m-2$, there exists a continuous operator $N_{q,t}^\eps : H^s_{0,q}(\Omega,e^{-t|z|^2},X^t) \to H^{s+2}_{0,q}(\Omega,e^{-t|z|^2},X^t) \cap \Dom(\Box^\ep_{q,t})$  so that
\begin{equation}
\label{eq:box_N}
\Box_{q,t}^\eps N_{q,t}^\eps u = u, \quad u\in H^s_{0,q}(\Omega,e^{-t|z|^2},X^t)
\end{equation}
and
\begin{equation}
\label{eq:N_box}
N_{q,t}^\eps \Box_{q,t}^\eps u = u, \quad u\in \Dom(\Box_{q,t}^\eps)\cap H^s_{0,q}(\Omega,e^{-t|z|^2},X^t).
\end{equation}
\end{prop}

\begin{proof}Using the notation of \cite{HaRa14}, we set $\X = H^1_{0,q}(\Omega,e^{-t|z|^2},X^t)\cap\Dom(\dbars_t)$. We may use Proposition  \ref{prop:basic estimate}
to see that for sufficiently large $|t|$, given any $u\in L^2_{0,q}(\Omega,e^{-t|z|^2})$, the map
$v\mapsto (u,v)_t$ is a continuous, conjugate linear functional on $\X$ since
\[
| (u,v)_t | \leq \|u\|_t \|v\|_t \leq \frac 1C \|u\|_t \big(Q_{t,\eps}(v,v)\big)^{1/2}
\]
where $C$ depends on $t$ but not on $\eps$. Thus, by the Riesz Representation Theorem, there exists a unique $N_{q,t}^\eps u \in \X$ so that
\[
(u,v)_t = Q_{t,\eps}(N_{q,t}^\eps u,v).
\]
Moreover, $N_{q,t}^\eps u \in \Dom(\Box_{q,t}^\eps)$ (this is standard -- see \cite{Str10}), as are the equalities \eqref{eq:box_N} and \eqref{eq:N_box}.

We now show that $N_{q,t}^\eps : H^s_{0,q}(\Omega,e^{-t|z|^2},X^t) \to H^{s+2}_{0,q}(\Omega,e^{-t|z|^2},X^t)$.
Since $\Dom(\dbars_t)$ is a closed subspace of $L^2_{0,q}(\Omega,e^{-t|z|^2})$ and $(H^1_0)_{0,q}(\Omega,e^{-t|z|^2},X^t)\subset\Dom(\dbars_t)$, it follows that
\[
(H^1_0)_{0,q}(\Omega,e^{-t|z|^2},X^t)\subset \X \subset H^1_{0,q}(\Omega,e^{-t|z|^2},X^t).
\]
Moreover, $Q_{t,\eps}(\cdot,\cdot)$ is strictly coercive over $\X$, so it follows from \cite[Theorem 3.13]{HaRa14} that $N_{q,t}^\eps : H^s_{0,q}(\Omega,e^{-t|z|^2},X^t) \to H^{s+2}_{0,q}(\Omega,e^{-t|z|^2},X^t)$
since $N_{q,t}^\eps u\in\X$ and $u\in H^s_{0,q}(\Omega,e^{-t|z|^2},X^t)$.
\end{proof}

We now introduce the concept of a tangential operator. We follow the notation of  \cite[\S4]{HaRa14}.
\begin{defn}\label{defn:tangential operator}
We call a first order differential operator $T^t$ \emph{weighted tangential} if there exists a vector field $T$ so that $T^t = T - tT|z|^2$ and $T\rho=0$. In other words,
the principal part of $T^t$ is tangential, and $(T^t)^*_t u= -Tu +O(u)$. If $\alpha$ is a multiindex and $(T^t)^\alpha = T^t_{\alpha_1}\cdots T^t_{\alpha_{|\alpha|}}$ where
each $T_{\alpha_j}^t$ is tangential, then we say that $(T^t)^\alpha$ is \emph{weighted tangential of order $|\alpha|$}.
\end{defn}

\begin{remark} It will be important that applying a (weighted) tangential derivative  preserves $\Dom(\dbarst)$. In order to see this, we fix an atlas of boundary charts and define the action of a tangential derivative to a form expressed
in the boundary coordinates to act componentwise. This will locally preserve $\Dom(\dbarst)$, and we can patch these together to obtain a global operator preserving $\Dom(\dbarst)$.  If we express our form in other coordinates, this will only introduce lower order terms with $C^{m-2}$ coefficients, and we will see that this causes no difficulty.  For more details, see
\cite[Section 5.2]{ChSh01} or the discussion in 2.3 of \cite{Str10}. When we differentiate with respect to tangential derivatives below, we are implicitly doing so in a way that preserves $\Dom(\dbarst)$.
\end{remark}

For a tangential operator $T^\alpha$, we will want to estimate $Q_{t,\eps}((T^t)^\alpha N^\eps_{q,t}, (T^t)^\alpha N^\eps_{q,t})$.  To do so, we will need to work with slightly smoother forms. To that end, we prove the following
density lemma that is a slight modification of \cite[Lemma 4.1]{HaRa15}.
\begin{lem}\label{lem:density} Let $\Omega\subset\C^n$ be a domain with a uniformly $C^m$ defining function, $m\geq3$,
and let $u\in H^1_{0,q}(\Omega,e^{-t|z|^2},X^t)\cap\Dom(\dbars_t)$. For any integer $k$ so that $2 \leq k \leq m-1$, there exists a sequence
$u_\ell\in C^k_{0,q}(\bar\Omega)\cap H^k_{0,q}(\Omega,e^{-t|z|^2},X^t) \cap \Dom(\dbars_t)$ converging to $u$ in the $H^1_{0,q}(\Omega,e^{-t|z|^2},X^t)$ norm.
\end{lem}

\begin{proof}
Let $\chi:\mathbb{R}\rightarrow\mathbb{R}$ be a smooth cutoff function satisfying $\chi(x)\equiv 0$ on $(-\infty,0]$ and $\chi(x)\equiv 1$ on $[1,\infty)$.
For $r>0$, let $u_r(z)=u(z)\chi\left(\frac{(r+1)^2-|z|^2}{(r+1)^2-r^2}\right)$.  Observe that any fixed number of derivatives of
$\chi\left(\frac{(r+1)^2-|z|^2}{(r+1)^2-r^2}\right)$ are uniformly bounded in $r$ and supported in $B(0,r+1)\backslash \overline{B(0,r)}$.
This means $u_r$ is supported in $B(0,r+1)$ and $u_r$ converges to $u$ in $H^1_{0,q}(\Omega,e^{-t|z|^2},X)$ as $r\rightarrow\infty$.
Let $\Omega_r$ be a bounded $C^m$ domain satisfying $\Omega_r\cap B(0,r+1)=\Omega\cap B(0,r+1)$.  By a straight forward adaptation of Lemma 4.1 in \cite{HaRa15},
we can build a sequence $\{u_{\ell,r}\}\subset C^k_{0,q}(\overline\Omega_r)\cap\Dom(\dbars_t)$ converging to $u_r$ on $\Omega_r$ with respect to $W^1_{0,q}(\Omega_r)$.
Multiplying again by our cutoff function gives us $u_{\ell,r}(z)\chi\left(\frac{(r+1)^2-|z|^2}{(r+1)^2-r^2}\right)\in C^k_{0,q}(\bar\Omega)\cap H^k_{0,q}(\Omega,e^{-t|z|^2},X^t) \cap \Dom(\dbars_t)$,
and we can extract a convergent subsequence by taking $r$ and $\ell$ sufficiently large.
\end{proof}

For the next lemma, we need to use special boundary charts.
Let $T_1,\dots, T_{n-1}$ be an orthonormal basis of $(1,0)$ vector fields near $\bd\Omega$ so that $T_j\rho=0$ on $\bd\Omega$. Let $T_n$ be the vector field so that
$T_n$ is orthogonal to $T_1,\dots,T_{n-1}$, $D_\nu := \Rre T_n = \frac{1}{\sqrt 2} \frac{\p}{\p\nu}$ and $T_\nu := \Imm T_n$ is tangential near $\bd\Omega$ and orthogonal to $T_1,\dots,T_{n-1}$.
Let $\bom^1,\dots,\bom^n$ be the dual basis. If $\bd\Omega$
has a uniformly $C^m$ defining function, then $\bom^j$ has coefficients (when expressed in the global coordinates $d\z^1,\dots,d\z^n$) that are uniformly $C^{m-1}$. Therefore,
$\dbar \bom^j$ has coefficients that are uniformly $C^{m-2}$.

In the special boundary chart, a $(0,q)$-form $u$ can be expressed as  $u = \sum_{J\in\I_q} u_J\, \bom^J$. Moreover, $u$ has
\[
\dbar u = \sum_{J\in\I_q}\sum_{k=1}^{n} \bar T_k u_J\, \bom^k\wedge\bom^J + O(u)
\quad\text{and}\quad
\vartheta^t u = -\sum_{I\in\I_{q-1}}\sum_{j=1}^n (T^t_j) u_{jI}\, \bom^I + O(u)
\]
where $u_{jI} = \sum_{J\in\I_q} \ep^{jI}_J u_J$. Note that in the formula for $\vartheta^t u$, the error term is $O(u)$, not $O_t(u)$. This is due to the fact that only the first order component of a weighted derivative satisfies the Leibniz formula, so, for example,
\[
T^t(fg) = g T^t f + f Tg.
\]
The normal derivative $D_\nu$ is defined for $z$ satisfying $\dist(z,\bd\Omega)<\rea(\bd\Omega)$ and has coefficients that are uniformly $C^{m-1}$.

\begin{lem}\label{lem:strengthening of normal deriv lemma from Straube}Let $\Omega\subset\C^n$ be a domain with a uniformly $C^m$ defining function, $m\geq 3$.
Let $u = \sum_{J\in\I_q} u_J\, \bom^J$ be a $(0,q)$-form defined near $\bd\Omega$. Let $J\in\I_q$.
\begin{enumerate}
\item  If $n\notin J$, then we can express $\frac{\p u_J}{\p\nu}$ as a linear combination
of coefficients of $\dbar u$, tangential derivatives of $u$, and $u$. The coefficients of the elements of $\dbar u$ and tangential derivatives of $u$ are uniformly $C^{m-1}$, and the coefficient of
$u$ is uniformly $C^{m-2}$.
\item
If $n\in J$, then the weighted normal derivative $D_\nu^t u_J$ can be expressed as a linear combination
of coefficients of $\vartheta^t u$, weighted tangential derivatives of $u$, and $u$. The coefficients of the elements of $\vartheta u$ and the weighted
tangential derivatives of $u$ are uniformly $C^{m-1}$, and the coefficient of $u$ is uniformly $C^{m-2}$.
\end{enumerate}
\end{lem}

\begin{proof} Investigating $\dbar u$, observe that  $k\not\in J$ if and only if $\bom^k\wedge\bom^J\neq 0$. Consequently, if $n\not\in J$, then the $\bom^n\wedge\bom^J$ component of $\dbar u$ is
\[
\Big(\dbar u\Big)_{\bom^n\wedge\bom^J} = \bar T_n u_J + \sum_{\atopp{J'\in\I_q}{J'\neq J}} \sum_{k=1}^n \ep^{kJ'}_{nJ} \bar T_k u_{J'} + O(u).
\]
Since $J'\neq J$, it follows that $k\neq n$ so that $\bar T_k$ is is a tangential vector field. Also, $\bar T_n u_J = D_{\nu}u_J - iT_\nu u_J$.
Note then that if $n\not\in J$, we have shown
\[
D_\nu u_J = \Big(\dbar u\Big)_{\bom^n\wedge\bom^J} - \sum_{\atopp{J'\in\I_q}{J'\neq J}} \sum_{k=1}^{n-1} \ep^{kJ'}_{nJ} \bar T_k u_{J'} + i T_\nu u_J+O(u).
\]
On the other hand when $n\in J$, we use $\vartheta^t u_J$ to control $D_\nu^t u_J$. Specifically, if $n\in J$ and $I = J\setminus\{n\}$, then $\ep^{nI}_J = (-1)^{q-1}$ and the $\bom^I$ component of
$\vartheta^t u$ is
\[
\big( \vartheta^t u \big)_{\bom^I} = -(-1)^{q-1}T^t_n u_J - \sum_{\atopp{J'\in\I_q}{J'\neq J}} \sum_{j=1}^{n-1} \ep^{jI}_{J'} T^t_j u_{J'} + O(u)
\]
Each of the nonzero weighted derivatives $\ep^{jI}_{J'} T^t_j$ are weighted tangential. This means
\[
(-1)^{q-1} D_\nu^t u_J = -\big( \vartheta^t u \big)_{\bom^I} - \sum_{\atopp{J'\in\I_q}{J'\neq J}} \sum_{j=1}^{n-1} \ep^{jI}_{J'} T^t_j u_{J'}-i(-1)^{q-1}T_\nu^t u_J + O(u)
\]
and the proof is complete.
\end{proof}
We would like to remove the dichotomy in Lemma \ref{lem:strengthening of normal deriv lemma from Straube}, namely, that some components are bounded with weighted tangential derivatives
and $\vartheta^t$ and others by unweighted tangential derivatives and $\dbar$.  However, we first record some technical lemmas about commutators of the various derivatives that appear.
\begin{lem}\label{lem:commutator of tangential/normal derivs}Let  $T^\alpha = T_{\alpha_1}\cdots T_{\alpha_\ell}$ be a tangential derivative of order $1 \leq \ell \leq m-1$
with coefficients that are uniformly $C^{\ell_1}$, $\ell \leq \ell_1$. If $X$ is a
first order differential operator with coefficients that are uniformly $C^{\ell_2}(\Om)$, $\ell \leq \ell_2$, then with $\ell_3 = \min\{\ell_1,\ell_2\}$
for every $\beta\subsetneq \alpha$, there exist first order operators $X_\beta$ with coefficients that are uniformly $C^{\ell_3-(\ell-|\beta|)}$ such that
\begin{equation}\label{eqn:commutator}
[T^\alpha ,X] = \sum_{\beta \subsetneq \alpha} T^\beta  X_\beta.
\end{equation}
\end{lem}

\begin{proof}The proof will follow from the computation that if $T = \sum_{j=1}^{2n} a_j \frac{\p}{\p x_j}$ and $X = \sum_{j=1}^n b_j \frac{\p}{\p x_j}$, then
\begin{equation}\label{eqn:T,X commutator}
[T,X] = \sum_{j,k=1}^{2n} \Big( a_j \frac{\p b_k}{\p x_j} - b_j \frac{\p a_k}{\p x_j}\Big) \frac{\p}{\p x_k}.
\end{equation}
Consequently, $[T,X]$ has coefficients that are uniform in $C^{\ell_3-1}$.

Let $T^\alpha = T_{\alpha_1}\cdots T_{\alpha_\ell}$. Then $[T^\alpha, X] = T^\alpha X - X T^\alpha$ and expanding the commutator in more detail, we observe
\begin{align}
&T^\alpha X - X T^\alpha = T^\alpha X - X T_{\alpha_1}\cdots T_{\alpha_\ell} = T^\alpha X - T_{\alpha_1} X T_{\alpha_2}\cdots T_{\alpha_\ell} + [X,T_{\alpha_1}] T_{\alpha_2} \cdots T_{\alpha_\ell} \nn\\
&=T^\alpha X - \big(T_{\alpha_1}T_{\alpha_2} X + T_{\alpha_1}[X,T_{\alpha_2}] + T_{\alpha_2}[X,T_{\alpha_1}] + \big[[X,T_{\alpha_1}],T_{\alpha_2}\big] \big) T_{\alpha_3}\cdots T_{\alpha_\ell}\nn \\
&= \sum_{\beta \subset \alpha}  T_{\beta_1} T_{\beta_2} \cdots T_{\beta_{|\beta|}} X_\beta \label{eqn:[T^alpha,X]}
\end{align}
where
\begin{equation}\label{eqn:X_beta}
X_\beta =  \big[ [[\cdots[[X,T_{(\alpha\setminus\beta)_1}],T_{(\alpha\setminus\beta)_2}], \cdots ],T_{(\alpha\setminus\beta)_{\ell-k-1}}],T_{(\alpha\setminus\beta)_{\ell-|\beta|}}\big]
\end{equation}
is an interated commutator of $X$ with $\ell-|\beta|$ tangential derivatives from $T^\alpha$ not included in $T^\beta$. We know that a commutator of two vectors fields with coefficients that are uniformly $C^k$
produces a vector field with coefficients that are uniformly $C^{k-1}$. Since the iterated commutator defining $X_\beta$ involves commuting $X$ with $\ell-|\beta|$ vector fields, $X_\beta$ is a vector field with uniformly $C^{\ell_3-(\ell-|\beta|)}$ coefficients.
\end{proof}

Observe that $[X_j^t,X_k^t]=\left(\left[\frac{\partial}{\partial x_k},\frac{\partial}{\partial x_j}\right]\right)^*_t=0$, so
\begin{equation}
\label{eqn: weighted deriv commutator}
[aX_j^t, b X_k^t] = a \frac{\p b}{\p x_j}X_k^t  - b\frac{\p a}{\p x_k} X_j^t.
\end{equation}

Using \eqref{eqn: weighted deriv commutator} to replace \eqref{eqn:T,X commutator}, we can
repeat the argument of Lemma \ref{lem:commutator of tangential/normal derivs} to prove the following, weighted derivative version.
\begin{lem}\label{lem:commutator of weighted tangential/normal derivs}Let  $T^\alpha = T_{\alpha_1}\cdots T_{\alpha_\ell}$ be a tangential derivative of order $1 \leq \ell \leq m-1$
with coefficients that are uniformly $C^{\ell_1}$, $\ell \leq \ell_1$. If $X$ is a
first order differential operator with coefficients that are uniformly $C^{\ell_2}(\Om)$, $\ell \leq \ell_2$, then with $\ell_3 = \min\{\ell_1,\ell_2\}$
for every $\beta\subsetneq \alpha$, there exists a first order weighted derivative $X_\beta^t$ with coefficients that are uniformly $C^{\ell_3-(\ell-|\beta|)}$ so that
\begin{equation}\label{eqn:[T^alpha^t,X^t]}
[(T^\alpha)^t ,X^t] = \sum_{\beta \subsetneq \alpha} (T^\beta)^t  X_\beta^t.
\end{equation}
\end{lem}
\hfill\qed

Finally, we investigate the situation when $X$ is not weighted but the tangential operators are weighted. The relevant commutator is $\left[X_j^t,\frac{\partial}{\partial x_k}\right]=\left[\frac{\partial}{\partial x_j}-2tx_j,\frac{\partial}{\partial x_k}\right]=2t\delta_{jk}$, so
\begin{equation}
\label{eqn:T^t, X commutator}
[aX_j^t, b \frac{\p}{\p x_k}] =  a \frac{\p b}{\p x_j}\frac{\p}{\p x_k}  - b\frac{\p a}{\p x_k} X_j^t +2abt\delta_{jk}.
\end{equation}

\begin{lem}\label{lem:commutator of weighted tangential/ unweighted normal derivs}Let  $T^\alpha = T_{\alpha_1}\cdots T_{\alpha_\ell}$ be a tangential derivative of order $1 \leq \ell \leq m-1$
with coefficients that are uniformly $C^{\ell_1}$, $\ell \leq \ell_1$. If $X$ is a
first order differential operator with coefficients that are uniformly $C^{\ell_2}(\Om)$, $\ell \leq \ell_2$, then with $\ell_3 = \min\{\ell_1,\ell_2\}$
for every $\beta\subsetneq \alpha$, there exist a first order weighted derivative $X_\beta^t$, a vector field $X_\beta'$, and a function $c_\beta$ so that
\begin{equation}\label{eqn:weighted commutator}
[(T^\alpha)^t ,X] = \sum_{\beta \subsetneq \alpha} (T^\beta)^t  \big(X_\beta^t +X_\beta' + tc_\beta\big).
\end{equation}
Moreover, $X_\beta^t$ and $X_\beta'$ have coefficients that are uniformly $C^{\ell_3 - (\ell-|\beta|)}$ and $c_\beta$ is a uniformly $C^{\ell_3-(\ell-|\beta|)+1}$ function.
\end{lem}

\begin{proof} The computation leading to (\ref{eqn:[T^alpha,X]}) and (\ref{eqn:X_beta}) was formal, so in this case, we have
\[
[X, (T^\alpha)^t] = \sum_{\beta \subset \alpha}  (T_{\beta_1})^t (T_{\beta_2})^t \cdots (T_{\beta_k})^t X_\beta''
\]
where
\[
X_\beta'' =  \big[ [[\cdots[[X,T_{(\alpha\setminus\beta)_1}^t],T_{(\alpha\setminus\beta)_2}^t], \cdots ],T_{(\alpha\setminus\beta)_{\ell-k-1}}^t],T_{(\alpha\setminus\beta)_{\ell-|\beta|}}^t\big]
\]
We need to understand the terms in $X_\beta''$, and the iterated commutator of length $\ell-|\beta|$. We know that there exist $Y_1$, $Y_1'$, and $c$ so that
\[
[X,T_{(\alpha\setminus\beta)_1}^t] =  Y_1^t +  Y_1' + tc
\]
where $Y_1^t$ and $Y_1'$ have coefficients that are uniformly $C^{\ell_3-1}$ and $c$ is uniformly $C^{\ell_3}$. Iterated commutators are linear in each of the components, so we can treat each piece
separately. The iterated commutator piece with $Y^t_1$ is handled with a repeated use of (\ref{eqn: weighted deriv commutator}) to produce a weighted derivative with coefficients that are uniformly
$C^{(\ell_3-1)-(\ell-|\beta|-1)}$, that is, uniformly $C^{\ell_3-(\ell-|\beta|)}$. The piece with $tc$ is also relatively straight forward to handle. Since $[c\cdot , Y^t] = -(Yc)\cdot$, we see that the commutator of
a function that is uniformly $C^{\ell'}$ with a weighted vector with uniformly $C^{\ell''}$ coefficients produces a function that is uniformly $C^{\min\{\ell',\ell''\}-1}$, hence the  iterated commutator of
length $\ell-|\beta|-1$ involving $tc$ produces a function $t c_1$ where $c_1$ is uniformly $\ell_3 - (\ell-|\beta|-1)$.

Thus, it remains to handle
\[
\big[ [[\cdots[[Y_1',T_{(\alpha\setminus\beta)_2}^t], \cdots ],T_{(\alpha\setminus\beta)_{\ell-k-1}}^t],T_{(\alpha\setminus\beta)_{\ell-|\beta|}}^t]\big],
\]
an iterated commutator of length $\ell-|\beta|-1$ of a vector field with uniformly $C^{\ell_3-1}$ coefficients with weighted tangential vector fields with uniformly $C^{\ell_2}$ coefficients. Repeating the
argument we just completed (i.e., induction on the length of the iterated commutator) shows that we can write this commutator in the form
\[
Y_2^t + Y_2' + t c_2
\]
where $Y_2^t$ is a weighted derivative with coefficients that are uniformly $C^{\ell_3-1 - (\ell-|\beta|-1)}$, i.e., uniformly $C^{\ell_3 - (\ell-|\beta|)}$. $Y_2'$ is a vector field, also with coefficients
that are uniformly $C^{\ell_3-(\ell-|\beta|)}$ and $c$ is a uniformly $C^{\ell_3-(\ell-|\beta|)+1}$ function.
\end{proof}

A corollary of the previous three lemmas is that if $\alpha$ is a multindex of length $k$ and $D^\alpha = D_{\alpha_1}\cdots D_{\alpha_k}$ where each $D_{\alpha_j}$ is either a vector field or weighted derivative that is tangential, respectively weighted tangential,
near $\bd\Omega$ with uniformly
$C^\ell$ coefficients, then for a vector field or weighted derivative $X$ with coefficients that are uniformly $C^{\ell_1}$, $k \leq\min\{ \ell,\ell_1\}$,
\[
[X,D^\alpha] = \sum_{\beta \subsetneq \alpha} D^\beta(X_\beta^t + X_\beta' + t c_\beta)
\]
where $X_\beta^t$ and $X_\beta'$ have uniformly $C^{\min\{\ell,\ell_1\}-(k-|\beta|)}$ coefficents and $c_\beta$ has uniformly\\ $C^{\min\{\ell,\ell_1\}-(k-|\beta|)+1}$ coefficients.

To show that $N^\ep_{q,t}$ is bounded in $H^k_{0,q}(\Omega, e^{-t|z|^2},X^t)$ with a constant independent of $\ep>0$, we need to
bound
\[
\| N^\ep_{q,t} f\|_{t,k,\Omega}^2 = \sum_{|\alpha|\leq k} \|(X^t)^\alpha N^\ep_{q,t} f\|_t^2.
\]	
The approach is to show that normal derivatives are controlled by $\dbar$, $\vartheta$, and tangential derivatives and therefore we only need to bound tangential derivatives to control the full Sobolev norm. This is accomplished in the next proposition.

In a neighborhood of each boundary point, we can define $\nabla_T$ to be the vector with components $(\Rre T_1,\Imm T_1, \dots, \Rre T_{n-1},\Imm T_{n-1}, T_\nu)$.  By a partition of unity, we can extend $\nabla_T$ to a uniform neighborhood of the boundary.  If we let $\nabla_T=\nabla_X$ on an interior set that is uniformly bounded away from the boundary, then we have a global gradient which differs from
$\nabla_X = (X_1,\dots,X_{2n})$ in that $\nabla_T$ contains only derivatives in the tangential directions on a uniform neighborhood of the boundary.  The following proposition extends Lemma \ref{lem:strengthening of normal deriv lemma from Straube} to higher order derivatives by showing that any $k$ derivatives of $v$ can be estimated in terms of $k-1$ derivatives of $\dbar v$, $k-1$ derivatives of $\vartheta ^t v$, $k$ weighted tangential derivatives of $v$, and lower order derivatives.
\begin{prop}\label{prop:controlling normal derivatives}
Let $v\in H^k_{0,q}(\Om,e^{-t|z|^2},X^t)$. Then there exist constants $C_k, C_{t,k}>0$ so that
\begin{equation}\label{eqn:k norm of v by dbar, dbars, tangent}
\|v\|_{t,k,\Om}^2 \leq C_k\Big( \|\dbar v\|_{t,k-1,\Om}^2 + \|\vartheta^t v\|_{t,k-1,\Om}^2 + \|(\nabla_T^t)^k v\|_t^2\Big) + C_{t,k}\|v\|_{t,k-1,\Om}^2.
\end{equation}
\end{prop}

\begin{proof}We first assume that in some uniform neighborhood of the boundary, our $k$th order derivative takes the form $(D_\nu^t)^\ell (T^t)^\alpha$ where $|\alpha|+\ell =k$.  Note that the ordering of the normal and tangential derivatives is irrelevant, since the commutators are of order at most $k-1$ and therefore bounded by \eqref{eqn:k norm of v by dbar, dbars, tangent}.  If $\ell=0$, then the result is trivial, so we will proceed by induction on $\ell$.

We next investigate the $\ell=1$, $\alpha=0$ case. Let $v$ be a $(0,q)$-form in $H^2_{0,q}(\Om, e^{-t|z|^2},X^t)\cap\Dom(\dbarst)$.
 We collect some estimates before we carefully write down the estimate of the normal direction.
Since $T_j$ is tangential,
\[
\|T_jv\|_t^2 = \|T_j^* v\|_t^2 + \big([T_j^*,T_j]v,v\big)_t,
\]
and it follows from \eqref{eqn:T^t, X commutator} and the fact that the commutator of tangential derivatives is tangential that
if $T_j$ has uniformly $C^k$ coefficients, then
\[
\|T_j v\|_t^2 = \|T_j^t v\|_t^2 + (c_j \cdot \nabla_T v,v)_t + (c_j'\cdot \nabla_T^t v, v)_t + C_t\|v\|_t^2
\]
and
\[
\|T_j^t v\|_t^2 = \|T_j v\|_t^2 + (d_j \cdot \nabla_T v,v)_t + (d_j'\cdot \nabla_T^t v, v)_t + C_t\|v\|_t^2
\]
where $c_j,c_j',d_j,d_j'$ are uniformly $C^{k-1}$ if $T_j$ has coefficients that are uniformly $C_k$. This means using a small constant/large constant argument and
absorbing terms, that there exists $C_{t}>0$
so that
\begin{equation}
\label{eq:tangential_derivatives_comparable_with_or_without_weight}
\|\nabla_T v\|_t^2 \leq 2\|\nabla_T^t v\|_t^2  + C_{t}\|v\|_t^2
\end{equation}
and
\[
\|\nabla_T^t v\|_t^2 \leq 2\|\nabla_T v\|_t^2  + C_{t}\|v\|_t^2.
\]

Write $v = v_1+ v_2$ where
\[
v_1 =  \sum_{\atopp{J\in\I_q}{n\in J}} v_J\,\bom^J
\qquad\text{and}\qquad
v_2 = \sum_{\atopp{J\in\I_q}{n\notin J}} v_J\, \bom^J,
\]
in suitable local coordinates near each boundary point.  In the interior, we let $v_1=v$ and $v_2=0$.  Turning to the normal derivatives themselves, we are now able to use Lemma \ref{lem:strengthening of normal deriv lemma from Straube} and establish
\begin{align*}
\| D_\nu v\|_t^2 + \|D_\nu^t v\|_t^2
&\leq 2\left(\| D_\nu v_1\|_t^2 + \|D_\nu^t v_1\|_t^2 + \| D_\nu v_2\|_t^2 + \|D_\nu^t v_2\|_t^2\right) \\
&\leq C(\|\vartheta^t v_1\|_t^2 + \|\nabla_T^t v_1\|_t^2 + \|\dbar v_2\|_t^2 + \|\nabla_T v_2\|_t^2 + \|v\|_t^2)\\
& \hspace{1in}+  \| D_\nu v_1\|_t^2 +  \| D_\nu^t v_2\|_t^2.
\end{align*}
Note that $\dbar$ acts tangentially on $v_1$ and $\vartheta^t$ acts tangentially on $v_2$, so we have
\begin{align*}
  \|\dbar v_2\|_t^2&\leq C(\|\dbar v\|_t^2 + \|\nabla_T v_1\|_t^2),\\
  \|\vartheta^t v_1\|_t^2&\leq C(\|\vartheta^t v\|_t^2 + \|\nabla_T^t v_2\|_t^2).
\end{align*}
Therefore,
\begin{align*}
\| D_\nu v\|_t^2 + \|D_\nu^t v\|_t^2
&\leq C\left(\|\vartheta^t v\|_t^2 + \|\nabla_T^t v_1\|_t^2 + \|\dbar v\|_t^2 + \|\nabla_T v_2\|_t^2\right) + C_t\|v\|_t^2\\
& \hspace{1in}+C\left( \|\nabla_T v_1\|_t^2 +  \| D_\nu v_1\|_t^2 + \|\nabla_T^t v_2\|_t^2 +  \| D_\nu^t v_2\|_t^2\right).
\end{align*}
By \cite[Proposition 3.5]{HaRa14},
\begin{align*}
  \|\nabla v_1\|_t^2&\leq C\|\nabla_X^t v_1\|_t^2 + C_t\|v_1\|_t^2,\\
  \|\nabla_X^t v_2\|_t^2&\leq C\|\nabla v_2\|_t^2 + C_t\|v_2\|_t^2.\\
\end{align*}
Therefore,
\begin{align*}
\| D_\nu v\|_t^2 + \|D_\nu^t v\|_t^2
&\leq C\left(\|\vartheta^t v\|_t^2 + \|\nabla_T^t v_1\|_t^2 + \|\dbar v\|_t^2 + \|\nabla_T v_2\|_t^2\right) + C_t\|v\|_t^2\\
& \hspace{1in}+C\left( \|\nabla_T^t v_1\|_t^2 +  \| D_\nu^t v_1\|_t^2 + \|\nabla_T v_2\|_t^2 +  \| D_\nu v_2\|_t^2\right).
\end{align*}
Using Lemma \ref{lem:strengthening of normal deriv lemma from Straube} again together with \eqref{eq:tangential_derivatives_comparable_with_or_without_weight}, we obtain
\[
\| D_\nu v\|_t^2 + \|D_\nu^t v\|_t^2
\leq C\Big(\|\dbar v\|_t^2 + \|\vartheta^t v\|_t^2  + \|\nabla_T^t v\|_t^2 \Big)  + C_{t} \|v\|_t^2.
\]

We have shown that there exist constants $C, C_t>0$ so that for every $v\in H^2_{0,q}(\Om, e^{-t|z|^2},X)$
\begin{equation}\label{eqn:normal derivatives benign}
\| D_\nu v\|_t^2 + \|D_\nu^t v\|_t^2
\leq C\Big(\|\dbar v\|_t^2 + \|\vartheta^t v\|_t^2  + \|\nabla_T^t v\|_t^2 \Big)  + C_{t} \|v\|_t^2.
\end{equation}
This proves the $k=1$ case.  For $k>1$, \eqref{eqn:normal derivatives benign} implies
\begin{align*}
&\| D_\nu (D_\nu^t)^{\ell-1} (T^t)^\alpha  h\|_t^2 + \|(D_\nu^t)^\ell (T^t)^\alpha h\|_t^2 \\
&\leq C\Big(\|\dbar (D_\nu^t)^{\ell-1} (T^t)^\alpha h \|_t^2 + \|\vartheta^t (D_\nu^t)^{\ell-1} (T^t)^\alpha h\|_t^2  + \|\nabla_T^t  (D_\nu^t)^{\ell-1} (T^t)^\alpha h\|_t^2 \Big)\\
  &\hspace{1in}+ C_{t} \|(D_\nu^t)^{\ell-1} (T^t)^\alpha h\|_t^2 \\
&\leq C\Big(\| (D_\nu^t)^{\ell-1}(T^t)^\alpha\dbar h \|_t^2 + \|(D_\nu^t)^{\ell-1} (T^t)^\alpha \vartheta^t h\|_t^2  + \|  (D_\nu^t)^{\ell-1}\nabla_T^t (T^t)^\alpha  h\|_t^2 \Big)  + C_{t,\ep} \|h\|_{t,k-1,\Om}^2 \\
&\leq C\Big(\| \dbar h \|_{t,k-1,\Om}^2 + \|\vartheta^t h\|_{t,k-1,\Om}^2  + \| (D_\nu^t)^{\ell-1} \nabla_T^t (T^t)^\alpha  h\|_t^2 \Big)  + C_{t,\ep} \|h\|_{t,k-1,\Om}^2.
\end{align*}
The third term can be estimated by our induction hypothesis on $\ell$, and we are done.
\end{proof}

Using Proposition \ref{prop:controlling normal derivatives}, normal derivatives are controlled by tangential derivatives and we see that to control $\|N^\ep_{q,t} f\|_{t,k,\Omega}^2$, it suffices to bound
$\|(T^t)^\alpha N^\ep_{q,t} f\|_{t,k,\Omega}$ where $|\alpha|=k$, $(T^t)^\alpha = T^t_{\alpha_1}\cdots T^t_{\alpha_k}$, and each $T^t_{\alpha_j}$ is tangential near $\bd\Omega$.  To do so, we must generalize \cite[Lemma 4.2]{HaRa15}. The issue is that the weighted terms are no longer benign in the sense that they cannot be treated separately as a lower order term from the first order part. As such, we investigate what the appropriate terms are
that we need for control of the $k$th derivatives.  Choose operators $D_1,\dots, D_{4n}$ so that near $\bd\Omega$, $\{D_j: 1 \leq j \leq 2n\}$ are tangential and span the tangential directions and away from the boundary span $\C^n$. Let
$\{D_j: 2n+1 \leq j \leq 4n\}$ be the weighted versions of first $2n$ vectors.  For example, when $1\leq j\leq 2n$ we may choose $D_j=\frac{\partial}{\partial x_j}-\frac{\partial\tilde\delta}{\partial x_j}\nabla\tilde\delta\cdot\nabla$ on a uniform neighborhood of the boundary and $D_j=\frac{\partial}{\partial x_j}$ away from this uniform neighborhood, with a suitable transition in between.  Thus, either set can be used to obtain the Sobolev norms of a form in $\Dom(\dbarst)$, as long as we are willing to pay a price
of multiplication by $t$ in the lower order terms. We also assume that each $D_j$ has uniformly $C^{m-1}$ coefficients.

Let $\alpha$ be a multiindex of length $k$ and $D^\alpha = D_{\alpha_1}\cdots D_{\alpha_k}$.
Let $u_\alpha = D^\alpha N^\ep_{q,t} f$. By the discussion after Lemmas \ref{lem:commutator of tangential/normal derivs}-\ref{lem:commutator of weighted tangential/ unweighted normal derivs}, we know that
\[
\big \|[D^\alpha ,\dbar^*_t] N^\ep_{q,t}f\big\|_t + \big\|[D^\alpha ,\dbar] N^\ep_{q,t}f\big\|_t + \big\| [D^\alpha,\nabla_X] N^\ep_{q,t} f\big\|_t
\leq C \big\|N^\ep_{q,t}f\big\|_{t,|\alpha|,\Om}+C_t\big\|N^\ep_{q,t}f\big\|_{t,|\alpha|-1,\Om}.
\]
This means
\begin{multline*}
Q_{t,\ep}(u_\alpha,u_\alpha)\leq \big|(D^\alpha \dbar N^\ep_{q,t} f,\dbar u_\alpha)_t\big|+\big|(D^\alpha \dbar^*_t N^\ep_{q,t} f,\dbar^*_t u_\alpha)_t\big|
+\ep\big|(D^\alpha \nabla_X  N^\ep_{q,t} f,\nabla_X u_\alpha)_t\big|\\
+\left(C\|N^\ep_{q,t}f\|_{t,|\alpha|,\Om}+C_t \|N^\ep_{q,t}f\|_{t,|\alpha|-1,\Om}\right)\sqrt{Q_{t,\ep}(u_\alpha,u_\alpha)}.
\end{multline*}
We would like to integrate by parts, but Proposition \ref{prop:elliptic regularity for N^q_eps} gives us only that
$\dbar u_\alpha$, $\dbarst u_\alpha$, $\nabla_X u_\alpha$ are in $H^1_{(0,q)}(\Omega)$. Therefore, we use Lemma \ref{lem:density} to approximate $u_\alpha$ by $u_\alpha^\ell \in H^{|\alpha|+1}_{(0,q)}(\Omega)\cap\Dom(\dbarst)$
and  integrate by parts and commute again to obtain
\begin{align*}
Q_{t,\ep}(u_\alpha,u_\alpha)
&\leq  \lim_{\ell\to\infty}\bigg[ \big|\big(D^\alpha \dbar N^\ep_{q,t} f,\dbar u_\alpha^\ell\big)_t\big|+\big|\big(D^\alpha \dbar^*_t N^\ep_{q,t} f,\dbar^*_t u_\alpha^\ell\big)_t\big|
+\ep\big|\big(D^\alpha \nabla_X  N^\ep_{q,t} f,\nabla_X u_\alpha^\ell\big)_t\big|\bigg]\\
&\hspace{1in}+\left(C\|N^\ep_{q,t}f\|_{t,|\alpha|,\Om}+C_t \|N^\ep_{q,t}f\|_{t,|\alpha|-1,\Om}\right)\sqrt{Q_{t,\ep}(u,u)} \\
&\leq \limsup_{\ell\to\infty} \bigg[\big|Q_{t,\ep}(N^\ep_{q,t} f,(D^\alpha)^* u_\alpha^\ell)\big|+\big|(\dbar N^\ep_{q,t} f,[(D^\alpha)^*,\dbar] u_\alpha^\ell)_t\big| \\
&\hspace{.5in}+\big|(\dbar^*_t N^\ep_{q,t} f,[(D^\alpha)^*,\dbar^*_t]u_\alpha^\ell)_t\big| +\ep\big|(\nabla_X  N^\ep_{q,t} f,[(D^\alpha)^*,\nabla_X] u_\alpha^\ell)_t\big|\bigg]\\
&\hspace{1in}+\left(C\|N^\ep_{q,t}f\|_{t,|\alpha|,\Om}+C_t \|N^\ep_{q,t}f\|_{t,|\alpha|-1,\Om}\right)\sqrt{Q_{t,\ep}(u_\alpha,u_\alpha)}.
\end{align*}
By definition,
\[
\limsup_{\ell\to\infty} Q_{t,\ep}(N^\ep_{q,t} f,(D^\alpha)^* u_\alpha^\ell)_t=\lim_{\ell\to\infty} (f,(D^\alpha)^* u_\alpha^\ell)_t=(D^\alpha  f,u_\alpha)_t.
\]

We have left to handle the commutator terms. They are estimated in the same fashion, and we show the estimate of $|(\dbar N^\ep_{q,t} f,[(D^\alpha)^*,\dbar] u_\alpha^\ell)_t|$. Recall that if $T$ is a tangential operator with uniformly $C^{m-1}$ coefficients, then
$T^* = -T^t + c_T$ where $c_T$ is a function that has uniformly $C^{m-2}$ coefficients. If $\alpha'$ is defined so that $D_{\alpha_j}^* = D_{\alpha'_{k-j+1}} + c_{\alpha_j}$, then
\[
(D^\alpha)^* = D_{\alpha_k}^* \cdots D_{\alpha_1}^* = (D_{\alpha'_1} + c_{\alpha_k})\cdots (D_{\alpha'_k} + c_{\alpha_1}) = D^{\alpha'} + \sum_{\beta'\subsetneq\alpha'} D^{\beta'} c_{\beta'}.
\]
Consequently, $[(D^\alpha)^*,\dbar]$ should be a form that we can control. In particular if\\ $h= \sum_{J\in\I_q} h_J\, d\z^J$, then
\[
[(D^\alpha)^*,\dbar] h = \sum_{J\in\I_q}\sum_{j=1}^n \Big[(D^\alpha)^*,\frac{\p}{\p\z_j}\Big] h_J\, d\z^j\wedge d\z^J,
\]
and for appropriate first order operators $X_{\beta',j}$ and functions $c_{\alpha',j}$
\begin{align*}
\Big[(D^\alpha)^*,\frac{\p}{\p\z_j}\Big]  & = \Big[D^{\alpha'} + \sum_{\beta'\subsetneq\alpha'} D^{\beta'} \tilde c_{\beta'}, \frac{\p}{\p\z_j}\Big] \\
&= \sum_{|\beta| \leq k-1} D^{\beta} X_{\beta,j} + c_{\alpha',j}.
\end{align*}
By an abuse of notation, we denote
\[
[(D^\alpha)^*,\dbar] h = \sum_{J\in\I_q}\sum_{j=1}^n \sum_{|\beta| \leq k-1}\left( D^{\beta} X_{\beta,j}h_J + c_{\alpha',j}h_J\right) d\z^j\wedge d\z^J
:= \sum_{|\beta|\leq k-1} D^\beta X_\beta h + c_{\alpha'} h.
\]
We now estimate
\begin{align*}
\lim_{\ell\to \infty} \big|(\dbar N^\ep_{q,t} f,[(D^\alpha)^*,\dbar] u_\alpha^\ell)_t\big|
&\leq \lim_{\ell\to \infty} \sum_{|\beta|\leq k-1} \big|(\dbar N^\ep_{q,t} f,(D^{\beta} X_{\beta} + c_{\alpha'}) u_\alpha^\ell)_t\big| \\
&= \sum_{|\beta|\leq k-1} \big|(\dbar N^\ep_{q,t} f,(D^{\beta} X_{\beta} + c_{\alpha'}) D^\alpha N^\ep_{q,t} f)_t\big| \\
&\leq \sum_{|\beta|\leq k-1} \big|\big((D^\beta)^* \dbar N^\ep_{q,t} f,X_\beta  D^\alpha N^\ep_{q,t} f \big)_t\big| +\|N^\ep_{q,t} f \|_{t,1,\Om} \|N^\ep_{q,t} f \|_{t,k,\Om}
\end{align*}
If $k=1$, we are done.  If $k>1$, we cannot pass the $X_\beta$ term to the other side, but we can commute $X_\beta$ with $D_{\alpha_1}$ and integrate by parts to bring the $D_{\alpha_1}$ term across the inner product. Specifically,
\begin{align*}
&\big|\big((D^\beta)^* \dbar N^\ep_{q,t} f,X_\beta  D^\alpha N^\ep_{q,t} f \big)_t\big| \\
&= \big|\big( D_{\alpha_1}^* (D^\beta)^* \dbar N^\ep_{q,t} f,X_\beta  D_{\alpha_2}\cdots D_{\alpha_k} N^\ep_{q,t} f \big)_t\big| + \big|\big( (D^\beta)^* \dbar N^\ep_{q,t} f, [X_\beta,D_{\alpha_1}]  D_{\alpha_2}\cdots D_{\alpha_k} N^\ep_{q,t} f \big)_t\big| \\
&\leq \big| \sum_{|\gamma| \leq |\beta|+1} \big(c_\gamma D^\gamma \dbar N_{q,t}^\ep f, X_\beta  D_{\alpha_2}\cdots D_{\alpha_k} N^\ep_{q,t} f \big)_t\big| + C \|N^\ep_{q,t} f\|_{t,k,\Om}^2 \\
&\leq C\Big(\sum_{|\gamma| \leq |\beta|+1} \sqrt{Q_{t,\ep}(D^\gamma N_{q,t}^\ep f, D^\gamma N_{q,t}^\ep f)}\|N^\ep_{q,t} f\|_{t,k,\Om} + \|N^\ep_{q,t} f\|_{t,k,\Om}^2\Big).
\end{align*}
Putting our estimates together, we have proven the following lemma.

\begin{lem}\label{lem:elliptic_regularization_key_step}
  Let $\Omega\subset \C^n$ satisfy the hypotheses of Theorem \ref{thm:closed range for dbar}.
  There exist constants $C>0$ independent of $t$ and $C_t>0$ depending on $t$ such that for any $f\in C^\infty_{(0,q)}(\Omega)$
  and operator $D^\alpha$ of order $k$ that is a composition of operators that are tangential near $\bd\Omega$ and each have coefficients that are at least uniformly $C^{m-1}$, we have
  \begin{align}
    &Q_{t,\ep}\big(D^\alpha  N^\ep_{q,t} f, D^\alpha N^\ep_{q,t} f\big) \leq C\big|(D^\alpha f, D^\alpha  N^\ep_{q,t} f\big)_t\big| \nn\\
    &+C\Big( \|N^\ep_{q,t}f\|^2_{t,|\alpha|,\Om} + \sum_{|\gamma|\leq k}\sqrt{Q_{t,\ep}(D^\gamma N_{q,t}^\ep f, D^\gamma N_{q,t}^\ep f)}\|N^\ep_{q,t} f\|_{t,k,\Om}\Big)   \label{eq:elliptic_regularization_key_step}
    +C_t\|N^\ep_{q,t}f\|_{t,|\alpha|-1,\Om}^2.
  \end{align}
\end{lem}

Summing over all $|\alpha|\leq k$, we use Lemma \ref{lem:elliptic_regularization_key_step} and obtain
\begin{multline*}
    \sum_{|\alpha|\leq k} Q_{t,\ep}\big(D^\alpha  N^\ep_{q,t} f, D^\alpha N^\ep_{q,t} f\big) \leq C\sum_{|\alpha|\leq k} \big|(D^\alpha f, D^\alpha  N^\ep_{q,t} f\big)_t\big| \\
    +C\Big( \|N^\ep_{q,t}f\|^2_{t,k,\Om} + \sum_{|\alpha|\leq k}\sqrt{Q_{t,\ep}(D^\alpha N_{q,t}^\ep f, D^\alpha N_{q,t}^\ep f)}\|N^\ep_{q,t} f\|_{t,k,\Om}\Big)  +C_t\|N^\ep_{q,t}f\|_{t,k-1,\Om}^2.
\end{multline*}
By using a small constant/large constant argument and absorbing terms, we see that
\begin{equation}\label{eqn:good est on Q}
\sum_{|\alpha|\leq k} Q_{t,\ep}\big(D^\alpha  N^\ep_{q,t} f, D^\alpha N^\ep_{q,t} f\big)
\leq C \|f\|_{t,k,\Om}^2 + C \|N^\ep_{q,t} f\|_{t,k,\Om}^2 + C_t \|N^\ep_{q,t} f\|_{t,k-1,\Om}^2.
\end{equation}

\begin{proof}[Proof of Theorem \ref{thm:closed range for dbar}]

The $s=0$ case for parts (i)-(vii) are the content of the \cite[Theorem 2.5]{HaRa17}.
Additionally, as a consequence of the interpolation theory developed for weighted Sobolev spaces on unbounded domains \cite{HaRa14}, it suffices to prove Theorem \ref{thm:closed range for dbar} when $s\in\N$.
We follow the argument from \cite[Section 6.4]{HaRa11}.

Proof of (v.a): Plugging \eqref{eqn:good est on Q} into Proposition \ref{prop:basic estimate}, we estimate that for any $\ep>0$
\begin{equation}\label{eqn:tangential bound}
\sum_{|\alpha|\leq k}\|D^\alpha  N^\ep_{q,t} f\|^2_{t} \leq \ep(\|f\|_{t,k,\Om}^2 + \|N^\ep_{q,t}f \|^2_{t,k,\Om} ) +C_t\|N^\ep_{q,t}f\|_{t,k-1,\Om}^2.
\end{equation}
for all $t$ sufficiently large.

As a consequence of Proposition \ref{prop:controlling normal derivatives}, bounding the tangential derivatives suffices to bound the normal derivatives.
Thus, \eqref{eqn:tangential bound} strengthens to
\[
\|N^\ep_{q,t}f\|^2_{t,k,\Om} \leq \ep(\|f\|_{t,k,\Om}^2 + \|N^\ep_{q,t}f\|^2_{t,k,\Om} ) +C_t\|N^\ep_{q,t}f\|_{t,k-1,\Om}^2.
\]
By choosing $\ep$ sufficiently small (which forces $t$ to be large), we can absorb terms and establish
\begin{equation}\label{eqn:continuity for N^delta}
\|N^\ep_{q,t}f\|^2_{t,k,\Om} \leq C\|f\|_{t,k,\Om}^2  +C_t\|N^\ep_{q,t}f\|_{t,k-1,\Om}^2
\end{equation}
where $C$ and $C_t$ are independent of $\ep$.

We now let $\ep\to 0$. We have shown that if $f\in W^{k}_{0,q}(\Omega)$, then $\{N_{q,t}^\ep f: 0 < \ep < 1\}$ is bounded in $W^{k}_{0,q}(\Omega)$. This means
there exists a sequence $\ep_\ell\to 0$ and $\tilde u \in W^{k}_{0,q}(\Omega)$ so that $N_{q,t}^{\ep_\ell} f \to \tilde u$ weakly in $W^{k}_{0,q}(\Omega)$. Consequently, if
$v\in (C^\infty_c)_{0,q}(\Omega)$, then
\[
\lim_{\ell\to\infty} Q_{t,\ep_\ell}(N_{q,t}^{\ep_\ell} f, v) = Q_t(\tilde u,v).
\]
Also,
\[
Q_{t,\ep_\ell}(N_{q,t}^{\ep_\ell} f, v) = (f,v) = Q_t(N_{q,t}f,v),
\]
so $N_{q,t}f = \tilde u$ and \eqref{eqn:continuity for N^delta} holds with $\ep=0$. Thus, $N_{q,t}$ is a continuous operator on $W^{k}_{0,q}(\Omega)$.

Proof of (v.b) and (v.d):
The continuity of $\dbar N_{q,t}$ and $\dbarst N_{q,t}$ in $W^{k}_{0,q}(\Omega)$ follows by choosing $t$ larger (possibly). Alternatively, we could
modify the argument of \cite[Section 6.5]{HaRa11}. As with $N_{q,t}$, we only need
to check that tangential derivatives are bounded.

The remaining items to show are (v.c) and (v.e),
the continuity of $N_{q,t}\dbar : W^{k}_{0,q-1}\cap\dom{\dbar} \to W^{k}_{0,q}$ and $N_{q,t}\dbar^*_t : W^{k}_{0,q+1}\cap\dom{\dbar^*_t} \to W^{k}_{0,q}$.
Although
the case $k=0$ was done at the end of Section \ref{sec:basic estimate, proof of L^2 case} (see also \cite[Theorem 4.3]{HaRa15}),
a comment is in order. The operator $N_{q,t}\dbars_t$ appears to require that the form be an element of $\dbars_t$. However, it follows from \cite[Lemma 3.2]{HaRa17} and the fact that
$(C^\infty_c)_{0,q+1}(\Om')$ is dense in $\Dom(\dbars_t)$ on bounded domains $\Om'$ \cite[Lemma 4.3.2]{ChSh01} that $(C^\infty_c)_{0,q+1}(\Om)$ forms are dense in $\Dom(\dbars_t)$ in $L^2_{0,q+1}(\Om,\vp)$. Consequently,
the fact that  $N_{q,t}\dbars_t = (\dbar N_{q,t})^*$ is a bounded operator on $L^2_{0,q+1}(\Om,\vp)$ means that we can define $N_{q,t}\dbars_t$ on $L^2_{0,q+1}(\Om,\vp)$ by density. Thus,
$N_{q,t}\dbars_t$ is defined for any form in $L^2_{0,q+1}(\Om)$, not just forms in $\Dom(\dbars_t)$. A similar argument also justifies writing $N_{q,t}\dbar$ applied to an arbitrary form in $L^2_{0,q-1}(\Om)$.
Since $N_{q,t}\dbar^*_t$ and $N_{q,t}\dbar$ both produce forms in $\Dom(\dbarst)$, it suffices to estimate the $W^{k}$ norm using only the special tangential derivatives $D^\alpha$. To that end,
observe that
\[
\dbar D^\alpha  N_{q,t} f = [\dbar, D^\alpha ] N_{q,t} f  +  D^\alpha  \dbar N_{q,t} f
\]
and
\[
\dbarst D^\alpha N_{q,t} f = [\dbarst, D^\alpha ] N_{q,t} f  +  D^\alpha  \dbarst N_{q,t} f.
\]
By Proposition \ref{prop:basic estimate}, given $\epsilon>0$ there exists $t>0$ so that
\[
\|N_{q,t} f \|_{t,k,\Om}^2 \leq C \sum_{|\alpha|\leq k} \|D^\alpha N_{q,t} f\|_t^2
\leq \ep\sum_{|\alpha|\leq k}\Big( \|\dbar D^\alpha N_{q,t}f\|_t^2 + \|\dbarst D^\alpha N_{q,t}f\|_t^2\Big).
\]
Since $f$ has smooth coefficients, by choosing $t$ larger (if necessary), we can use a small constant/large constant argument and estimate
\[
\|N_{q,t} f \|_{t,k,\Om}^2 \leq \ep \sum_{|\alpha|\leq k}\Big( \|D^\alpha  \dbar N_{q,t}f\|_t^2 + \|D^\alpha  \dbarst N_{q,t}f\|_t^2\Big) + C_t \|N_{q,t} f\|_{t,k-1,\Om}^2.
\]

Next, suppose that $f = \dbar^*_t g$ for a $(0,q+1)$-form in $g\in\dom{\dbar^*_t}$ with smooth coefficients. Then by induction and \cite[(22)]{HaRa11},
\[
\|N_{q,t}\dbar^*_t g \|_{t,k,\Om}^2
\leq \ep \sum_{|\alpha|\leq k} \| D^\alpha \dbar N_{q,t} \dbar^*_t g\|_t^2 + C_t \|g\|_{t,k-1,\Om}^2.
\]

We now handle the $\| D^\alpha \dbar N_{q,t} \dbar^*_t g\|_t^2$ term. Since $D^\alpha \dbar N_{q,t} \dbar^*_t g\in\Dom(\dbarst)$, it follows that
\begin{align*}
&\| D^\alpha \dbar N_{q,t} \dbar^*_t g\|_t^2
= \big(D^\alpha  N_{q,t} \dbar^*_t g , \dbarst D^\alpha \dbar N_{q,t} \dbar^*_t g \big)_t + \big([D^\alpha,\dbar] N_{q,t} \dbar^*_t g ,  D^\alpha \dbar N_{q,t} \dbar^*_t g \big)_t \\
&\leq \big|\big(D^\alpha  N_{q,t} \dbar^*_t g , D^\alpha \dbarst \dbar N_{q,t} \dbar^*_t g \big)_t\big| + O\Big(\|N_{q,t}\dbar^*_t g\|_{t,|\alpha|,\Omega}\|\dbar N_{q,t} \dbar^*_t g\|_{t,|\alpha|,\Omega}\Big) \\
&\hspace{3in}+ O_t\Big(\|N_{q,t}\dbar^*_t g\|_{t,|\alpha|-1,\Omega}\|\dbar N_{q,t} \dbar^*_t g\|_{t,|\alpha|,\Omega}\Big)\\
&\leq \big|\big(D^\alpha  N_{q,t} \dbar^*_t g , D^\alpha  \dbar^*_t g \big)_t\big| + l.c.\|N_{q,t}\dbar^*_t g\|_{t,|\alpha|,\Omega} + s.c.\|\dbar N_{q,t} \dbar^*_t g\|_{t,|\alpha|,\Omega} + C_t\|N_{q,t}\dbar^*_t g\|_{t,|\alpha|-1,\Omega} .
\end{align*}
Thus,
\[
\sum_{|\alpha|\leq k}\| D^\alpha \dbar N_{q,t} \dbar^*_t g\|_t^2 \leq 2\sum_{|\alpha|\leq k}\big|\big(D^\alpha  N_{q,t} \dbar^*_t g , D^\alpha  \dbar^*_t g \big)_t\big| + C\|N_{q,t}\dbar^*_t g\|_{t,|\alpha|,\Omega}^2 + C_t\|N_{q,t}\dbar^*_t g\|_{t,|\alpha|-1,\Omega}.
\]
Next,
\begin{multline*}
\sum_{|\alpha|\leq k}\big(D^\alpha  N_{q,t} \dbar^*_t g , D^\alpha  \dbar^*_t g \big)_t
= \sum_{|\alpha|\leq k}\big(D^\alpha  \dbar N_{q,t} \dbar^*_t g , D^\alpha  g \big)_t \\
+ O\Big(\|\dbar N_{q,t} \dbar^*_t g\|_{t,k,\Omega}\|g\|_{t,k,\Omega}\Big) +  O_t\Big(\|\dbar N_{q,t} \dbar^*_t g\|_{t,k,\Omega}\|g\|_{t,k-1,\Omega}\Big).
\end{multline*}
Thus, by absorbing terms after a small constant/large constant argument, we have
\[
\sum_{|\alpha|\leq k}\| D^\alpha \dbar N_{q,t} \dbar^*_t g\|_t^2 \leq  C\|g\|_{t,k,\Omega}^2  + C_t\|g\|_{t,k-1,\Omega}^2 + C\|N_{q,t}\dbar^*_t g\|_{t,k,\Omega}^2.
\]
Finally, by choosing $\ep$ sufficiently small to absorb the $\|N_{q,t}\dbar^*_t g\|_{t,|\alpha|,\Omega}^2$ terms, we have proven
\[
\|N_{q,t}\dbar^*_t g \|_{t,k,\Om}^2
\leq \ep \| g\|_{t,k,\Om}^2 + C_t \|g\|_{t,k-1,\Om}^2.
\]

The argument to prove
\[
\|N_{q,t}\dbar g \|_{t,k,\Om}^2
\leq \ep \| g\|_{t,k,\Om}^2 + C_t \|g\|_{t,k-1,\Om}^2
\]
is similar.

Proof of (v.g): It suffices to prove the $s=0$ case for smooth, compactly supported $u \in L^2_{0,q}(\Om,e^{-t|z|^2})$.
Then $\dbars_t\dbar N_{q,t}u$ is well defined, by (iv). Since $\Ran \dbar \perp \Ran \dbars_t$, it follows that
\[
\| \dbars_t \dbar N_{q,t} u \|_t^2 =  \big( \Box_{q,t} N_{q,t} u,  \dbars_t \dbar N_{q,t} u\big)_t
=  \big( u,  \dbars_t \dbar N_{q,t} u\big)_t \leq \|u\|_t \| \dbars_t \dbar N_{q,t} u \|_t,
\]
from which the $s=0$ case follows. The argument for $\dbar \dbars_t N_{q,t} u$ is identical.
For $s>0$, we use proof by induction and will show that for any multiindex $\beta$ of length $s$,
\[
\|X^\beta \dbar\dbars_t N_{q,t} u \|_t^2 + \|X^\beta \dbars_t\dbars N_{q,t} u \|_t^2 \leq C_t \|u\|_{t,s,\Om}^2.
\]
The estimate will follow from the argument of the $s=0$ case couple with the following estimate.
\begin{align*}
&\big(X^\beta \dbars_t \dbar N_{q,t}u, X^\beta \dbar\dbars_t N_{q,t} u\big)_t\\
&= \big([X^\beta,\dbars_t] \dbar N_{q,t}u, X^\beta \dbar\dbars_t N_{q,t} u\big)_t + \big(\dbars_t X^\beta \dbar N_{q,t} u, [X^\beta,\dbar]\dbars_t N_{q,t} u\big)_t \nn \\
&= \big([X^\beta,\dbars_t] \dbar N_{q,t}u, X^\beta \dbar\dbars_t N_{q,t} u\big)_t + \big([\dbars_t, X^\beta] \dbar N_{q,t} u, [X^\beta,\dbar]\dbars_t N_{q,t} u\big)_t\\
&\hspace{1in}+\big(X^\beta \dbars_t \dbar N_{q,t} u, [X^\beta,\dbar]\dbars_t N_{q,t} u\big)_t \nn \\
&\leq \|\dbar N_{q,t}u\|_{t,s,\Om} \|X^\beta \dbar\dbars_t N_{q,t} u\|_t +  \|X^\beta \dbars_t\dbar N_{q,t} u\|_t \|\dbars_t N_{q,t}u\|_{t,s,\Om}\\
&\hspace{1in}+ \|\dbar N_{q,t}u\|_{t,s,\Om}\big( \|\dbars_t N_{q,t}u\|_{t,s,\Om} + C_t \|\dbars_t N_{q,t}u\|_{t,s-1,\Om}\big)
\end{align*}
Indeed,
\begin{align*}
&\|X^\beta \dbar\dbars_t N_{q,t} u \|_t^2 + \|X^\beta \dbars_t\dbars N_{q,t} u \|_t^2
=\big( X^\beta u, X^\beta \dbar\dbars_t N_{q,t} u \big)_t + \big( X^\beta u, X^\beta \dbars_t\dbar N_{q,t} u \big)_t \\
&- \big(X^\beta \dbars_t \dbar N_{q,t} u, X^\beta \dbar\dbars_t N_{q,t} u\big)_t - \big(X^\beta  \dbar\dbars_t N_{q,t} u, X^\beta \dbars_t\dbar N_{q,t} u\big)_t \\
&\leq \|X^\beta u \|_t\big(\|X^\beta \dbar\dbars_t N_{q,t} u \|_t + \|X^\beta \dbars_t\dbars N_{q,t} u \|_t\big)  \\
&+  \|\dbar N_{q,t}u\|_{t,s,\Om}\big( \|X^\beta \dbar\dbars_t N_{q,t} u\|_t + \|\dbars_t N_{q,t}u\|_{t,s,\Om} + C_t \|\dbars_t N_{q,t}u\|_{t,s-1,\Om}\big) \\
&+ \|\dbars_t N_{q,t}u\|_{t,s,\Om}\big(\|X^\beta \dbars_t\dbar N_{q,t} u\|_t +  \|\dbar N_{q,t}u\|_{t,s,\Om} \big).
\end{align*}
From this calculation, it follows that
\[
\|X^\beta \dbar\dbars_t N_{q,t} u \|_t^2 + \|X^\beta \dbars_t\dbars N_{q,t} u \|_t^2 \leq C \|u \|_{t,s,\Om}^2 + C_{s,t}\|u \|_{t,s-1,\Om}^2
\]
and consequently half of (v.g) is proven.

Proof of the remaining case of (v.g):
By Proposition \ref{prop:controlling normal derivatives}, we claim that it suffices to consider derivatives $X^\beta$ that are tangential near $\bd\Om$. Indeed,
\[
\|\dbars_t N_{q,t} \dbar u\|_{t,k,\Om}^2 \leq C_k\big( \big\|(\nabla^t_T)^k \dbars_tN_{q,t} \dbar u\big\|_t^2 + \|\dbar\dbars_{q,t}N_{q,t}\dbar u\|_{t,k-1,\Om}^2\big) + C_{t,k} \|\dbars_t N_{q,t} \dbar u\|_{t,k-1,\Om}^2.
\]
However,
\[
\|\dbar\dbars_{q,t}N_{q,t}\dbar u\|_{t,k-1,\Om}^2 = \|\dbar u\|_{t,k-1,\Om}^2 \leq \|u\|_{t,k,\Om}^2.
\]
Therefore, let $T^\beta$ be an order $k$ operator that is tangential near $\bd\Om$. The reason that tangential operators are important here is that if $f\in\Dom(\dbars_t)$, then so is $T^\beta f$.
\begin{align*}
&\|T^\beta \dbars_t N_{q,t} \dbar u\|_{t,k,\Om}^2\\
&= \big( T^\beta \dbar u, T^\beta N_{q,t}\dbar u\big)_t + \big(T^\beta\dbars_t N_{q,t}\dbar u, [T^\beta,\dbars_t]N_{q,t}\dbar u\big)_t + \big([\dbar,T^\beta]\dbars_t N_{q,t}\dbar u, T^\beta N_{q,t}\dbar u\big)_t \\
&=  \big( T^\beta u, T^\beta \dbars_t N_{q,t}\dbar u\big)_t + \big(T^\beta\dbars_t N_{q,t}\dbar u, [T^\beta,\dbars_t]N_{q,t}\dbar u\big)_t + \big([\dbar,T^\beta]\dbars_t N_{q,t}\dbar u, T^\beta N_{q,t}\dbar u\big)_t \\
&+ \big([T^\beta,\dbar]u,T^\beta N_{q,t}\dbar u\big)_t + \big(T^\beta u, [\dbars_t,T^\beta]N_{q,t}\dbar u\big)_t.
\end{align*}
Part (a) now follows from the Cauchy-Schwarz inequality, followed by a small constant/large constant argument and term absorbtion.

Proof if (i)-(iv): This group of results follows from the continuity of the $\dbar$-Neumann operator and the canonical solutions operators.

Proof of (vii): We have established the estimates for Kohn's weighted theory, so solvability in $C^\infty$ proceeds using standard arguments. See, for example, \cite[Section 6.8]{HaRa11}.
\end{proof}

\begin{rem}Our argument can also be used to established the following estimate: for $k\geq 1$, there
exists $T_k>0$ so that if $t \geq T_k$, then
there exist constants $C_k, C_{k,t}>0$ where $C_s$ does not depend on $t$ and so that for any $u\in H^k_{0,q}(\Om,e^{-t|z|^2},X) \cap \Dom(\dbars_t)$ with $\dbar u \in H^k_{0,q+1}(\Om,e^{-t|z|^2},X)$ and
$\dbars_t u \in H^k_{0,q-1}(\Om,e^{-t|z|^2},X)$, the inequality
\[
\|u\|_{t,k,\Om}^2 \leq C_k\big( \|\dbar u\|_{t,k,\Om}^2 + \|\dbars_t u\|_{t,k,\Om}^2\big) + C_{k,t}\|u\|_{t,k-1,\Om}^2.
\]
holds.
The key to show this estimate is Proposition \ref{prop:controlling normal derivatives} which allows us to estimate the norm in $H^k_{0,q}(\Om,e^{-t|z|^2},X)$ for forms in $\Dom(\dbar)\cap\Dom(\dbars_t)$ with derivatives that are
tangential near $\bd\Om$.
Consequently, the argument is no different than the $\dbarb$ case which we established in \cite[\S6.3]{HaRa11}. The result is a matter of careful integration by parts. It is not immediate that this is a closed range estimate for
$\dbar$ (or $\dbars_t$) since $\dbars_t$ is the $L^2$-adjoint of $\dbar$, not the $H^k$-adjoint.
\end{rem}

\bibliographystyle{alpha}
\bibliography{mybib1_27_17}

\end{document}